\documentclass[12pt, reqno]{amsart}

\usepackage{
amsmath,
amssymb,
amsthm,
mathrsfs,
amsfonts,
ytableau,
enumitem,
upgreek,
soul,
mathtools,
shuffle,
array,
caption
}

\usepackage{thmtools}
\usepackage[nocompress]{cite}
\usepackage[linktocpage]{hyperref}

\hypersetup{
    colorlinks=true,
    linkcolor = blue,
    citecolor=magenta,
    urlcolor=cyan,
}

\usepackage[capitalise, nameinlink]{cleveref}

\crefname{equation}{}{}
\crefname{figure}{{\sc Figure}}{{\sc Figure}}

\numberwithin{equation}{section}
\numberwithin{figure}{section}
\numberwithin{table}{section}

\newtheorem{theorem}{Theorem}[section]
\newtheorem{proposition}[theorem]{Proposition}
\newtheorem{lemma}[theorem]{Lemma}

\newtheorem*{claim*}{Claim}

\theoremstyle{definition}
\newtheorem{algorithm}[theorem]{Algorithm}
\newtheorem{example}[theorem]{Example}
\newtheorem{definition}[theorem]{Definition}
\newtheorem{remark}[theorem]{Remark}

\usepackage[euler]{textgreek}
\usepackage{tikz}

\ytableausetup{mathmode, boxsize=1.2em}

\def\Z{\mathbb Z}
\def\C{\mathbb C}

\newcommand{\nc}{\newcommand}

\nc{\SG}{\mathfrak{S}}
\nc{\comp}{\mathrm{comp}}
\nc{\Des}{\mathrm{Des}}
\nc{\DesL}{\mathrm{Des}_{L}}
\nc{\set}{\mathrm{set}}
\nc{\ch}{\mathrm{ch}}
\nc{\Sym}{\mathrm{Sym}}
\nc{\QSym}{\mathrm{QSym}}
\nc{\SSYT}{\mathrm{SSYT}}
\nc{\SYT}{\mathrm{SYT}}
\nc{\SVT}{\mathrm{SVT}}
\nc{\StdSVT}{\mathrm{StdSVT}}
\nc{\wt}{\mathrm{wt}}
\nc{\sh}{\mathrm{sh}}
\nc{\std}{\mathrm{std}}
\nc{\tyd}{\mathtt{yd}}
\nc{\row}{\mathsf{row}}
\nc{\col}{\mathsf{col}}
\nc{\sfread}{\mathsf{read}}
\nc{\Hmod}{H_n(0)\text{-mod}}

\nc{\calB}{\mathcal{B}}
\nc{\calE}{\mathcal{E}}
\nc{\calI}{\mathcal{I}}
\nc{\bfF}{\mathbf{F}}
\nc{\bfG}{\mathbf{G}}
\nc{\scrG}{\mathscr{G}}
\nc{\sfB}{\mathsf{B}}
\nc{\ra}{\rightarrow}
\nc{\bfx}{\mathbf{x}}
\nc{\sfx}{\mathsf{x}}
\nc{\bfc}{\mathbf{c}}

\definecolor{purple}{rgb}{0.44, 0.0, 1.0}

\definecolor{yhblue}{rgb}{0,0,0.6}

\newenvironment{red}{\relax\color{red}}{\hspace*{.5ex}\relax}
\newenvironment{blue}{\relax\color{blue}}{\hspace*{.5ex}\relax}
\newenvironment{green}{\relax\color{wsgreen}}{\hspace*{.5ex}\relax}
\newenvironment{magenta}{\relax\color{magenta}}{\hspace*{.5ex}\relax}

\nc{\ber}{\begin{red}}
\nc{\er}{\end{red}}
\nc{\beb}{\begin{blue}}
\nc{\eb}{\end{blue}}
\nc{\bema}{\begin{magenta}}
\nc{\ema}{\end{magenta}}
\nc{\begr}{\begin{green}}
\nc{\egr}{\end{green}}

\title[Weak Bruhat interval modules for stable Grothendieck polynomials]
{Weak Bruhat interval modules of the 0-Hecke algebras for stable Grothendieck polynomials}

\author[Y.-H. Kim]{Young-Hun Kim}
\address{Department of Mathematics, Seoul Women’s University, Seoul 01797, Republic of Korea}
\email{yhkim@swu.ac.kr, ykim.math@gmail.com}



\thanks{
This work was supported by the National Research Foundation of Korea(NRF) grant
funded by the Korea government(MSIT) (No. RS-2026-25595299).
}

\begin{document}

\begin{abstract}
For a partition $\lambda$, let $G_\lambda^{(\beta)}$ be the stable $\beta$-Grothendieck polynomial attached to $\lambda$. Each homogeneous component of the $\beta = 1$ specialization $G_\lambda^{(1)}$ is Schur-positive and hence positive in the fundamental basis of quasisymmetric functions. For $m\ge|\lambda|$, let $G_{\lambda,m}^{(1)}$ be the homogeneous degree $m$ component of $G_\lambda^{(1)}$. In this paper, we first give a direct proof of an expansion of $G_{\lambda,m}^{(1)}$ in the fundamental basis in terms of standard set-valued tableaux. We then use these tableaux as a basis to define a module of the $0$-Hecke algebra and show that the quasisymmetric characteristic of the resulting module is $G_{\lambda,m}^{(1)}$. We further show that this module decomposes as a direct sum of weak Bruhat interval modules.
\end{abstract}

\maketitle

\section{Introduction}\label{Sec: introduction}

Grothendieck polynomials, indexed by permutations, were originally defined by Lascoux and Sch\"utzenberger~\cite{LS82} and serve as polynomial representatives for $K$-theoretic Schubert classes of type $A$ flag varieties.
Fomin and Kirillov~\cite{FK94} introduced a $\beta$-deformation of Grothendieck polynomials and studied the corresponding stable limits, now called stable $\beta$-Grothendieck polynomials.
These stable limits are symmetric formal power series indexed by permutations.
For a Grassmannian permutation, that is, a permutation with at most one descent, the corresponding stable limit is naturally indexed by a partition.
For a partition $\lambda$, let $G_\lambda^{(\beta)}$ denote the stable $\beta$-Grothendieck polynomial attached to $\lambda$.
In \cite{Buch02}, Buch provided a set-valued tableau formula for $G_\lambda^{(-1)}$.
Buch's formula extends to $G_\lambda^{(\beta)}$, with the exponent of $\beta$ recording the number of entries in excess of $|\lambda|$; see, for example, \cite{McNamara06,PS19,MPS21}.

The specializations of $G_\lambda^{(\beta)}$ at $\beta=0$ and $\beta=-1$ are classical objects in Schubert calculus.
Indeed, $G_\lambda^{(\beta)}$ is the Schur function $s_\lambda$ at $\beta=0$, and it serves as a representative for the $K$-theoretic Schubert class attached to $\lambda$ in a suitable Grassmannian at $\beta=-1$.
In this paper, we focus on the $\beta=1$ specialization $G_\lambda^{(1)}$.
In this case, for each $m\ge|\lambda|$, the homogeneous degree $m$ component of $G_\lambda^{(1)}$ is Schur-positive, and the set-valued tableau formula is unsigned.
For convenience, we write $G_\lambda:=G_\lambda^{(1)}$ and denote its homogeneous degree $m$ component by $G_{\lambda,m}$.
For more information, see \cref{subsec: groth prelim}.

The $0$-Hecke algebra $H_n(0)$ is obtained from the Iwahori--Hecke algebra $H_n(q)$ by specializing $q$ to $0$.
Norton~\cite{Norton79} classified all irreducible and projective indecomposable $H_n(0)$-modules up to isomorphism.
Duchamp, Krob, Leclerc, and Thibon~\cite{DKLT96} introduced a ring isomorphism
\[
\ch:\bigoplus_{n\ge 0} G_0(\Hmod) \rightarrow \QSym,
\quad
[\bfF_\alpha]\mapsto F_\alpha
\]
called the quasisymmetric characteristic.
For the notation used here, see \cref{subsec: qsym prelim,subsec: 0hecke prelim}.
In view of this correspondence, many $0$-Hecke modules with important quasisymmetric characteristics have been constructed by defining suitable actions on combinatorial models rather than simply taking direct sums of irreducible modules.
Such modules have been constructed for a variety of quasisymmetric functions, including dual immaculate functions \cite{BBSSZ15}, quasisymmetric Schur functions \cite{TW15}, extended Schur functions \cite{Searles20}, Young row-strict quasisymmetric Schur functions \cite{BS22}, Young quasisymmetric Schur functions \cite{CKNO22b, LO25}, and genomic Schur functions \cite{KY24}.

Weak Bruhat interval modules were introduced by Jung, Kim, Lee, and Oh~\cite{JKLO22} to provide a uniform framework for studying $0$-Hecke modules.
Indeed, the authors investigated structural properties of these modules and showed that every indecomposable direct summand of several previously constructed $0$-Hecke modules is isomorphic to a weak Bruhat interval module.
Since then, weak Bruhat interval modules have been actively studied.
In \cite{CKO24}, it was shown that every weak Bruhat interval module can be realized as a poset module.
In the same year, a family of weak Bruhat interval modules whose quasisymmetric characteristics are skew Schur functions was studied from the viewpoint of poset modules \cite{KLO24}.
Beyond type $A$, the definition of weak Bruhat interval modules was extended to all finite Coxeter types in \cite{BS25}, where projective covers and injective hulls were also studied for certain families of these modules.
More recently, a classification of weak Bruhat interval modules up to isomorphism was provided for arbitrary Coxeter types in \cite{LYY26}.
For further results related to weak Bruhat interval modules, see also \cite{DS26,Searles25,KS26}.

The purpose of this paper is to provide a nice representation-theoretic interpretation of each homogeneous component $G_{\lambda,m}$ of $G_\lambda$.
It is natural to first ask whether $G_{\lambda,m}$ is the quasisymmetric characteristic of an indecomposable $H_m(0)$-module.
However, it is not always possible to construct such a module; see \cref{rem: indecomposable module}.
We therefore construct an $H_m(0)$-module that has quasisymmetric characteristic $G_{\lambda,m}$ and decomposes as a direct sum of weak Bruhat interval modules.
Below, we describe our results in more detail.
Hereafter, let $\lambda$ be a partition and let $m$ be an integer with $m\ge|\lambda|$.

In \cref{sec: svt action}, we first define a reading word $w_T$ for set-valued tableaux and use it to define their standardization (\cref{alg:standardization}).
With these notions, we give a direct proof of the $F$-expansion
\[
G_{\lambda,m}
=\sum_{T\in\StdSVT_m(\lambda)}F_{\comp(\Des(T))}
\]
in terms of standard set-valued tableaux (\cref{thm: stable groth F expansion}).
Here, $\StdSVT_m(\lambda)$ is the set of standard set-valued tableaux of shape $\lambda$ with $m$ entries.
We next define an $H_m(0)$-action on the $\C$-span of $\StdSVT_m(\lambda)$ and denote the resulting module by $\scrG_{\lambda;m}$ (\cref{thm: svt action is 0hecke}).
Using the $F$-expansion above, we prove in \cref{thm: characteristic} that 
\[
\ch([\scrG_{\lambda;m}])=G_{\lambda,m}.
\]
We also study the decomposition of $\scrG_{\lambda;m}$.
To do this, we define an equivalence relation on $\StdSVT_m(\lambda)$ (\cref{def: column equivalence}).
Let $\calE_{\lambda,m}$ denote the set of equivalence classes.
In \cref{thm: component direct sum}, we prove that for each $E \in \calE_{\lambda,m}$, the $\C$-span of $E$ is closed under the $H_m(0)$-action.
Consequently,
\[
\scrG_{\lambda;m} = \bigoplus_{E\in\calE_{\lambda,m}}\scrG_E,
\]
where $\scrG_E$ is the $H_m(0)$-submodule of $\scrG_{\lambda;m}$ whose underlying space is the $\C$-span of $E$.

\cref{sec: WBIM decomposition} is devoted to proving that each $\scrG_E$ is isomorphic to a weak Bruhat interval module.
We first construct two distinguished tableaux $T_E$ and $T'_E$ in $E$ (\cref{alg: extremal tableaux}).
In \cref{thm: unique source sink}, we prove that $T_E$ and $T'_E$ are the unique source and sink tableaux in $E$, respectively.
For the definitions of source and sink tableaux, see \cref{def: source sink}.
Then we assign a permutation $\sfread(T)$ to each $T \in E$ and define a partial order $\preceq_E$ on $E$ in terms of the $H_m(0)$-action.
With these preparations, we prove that the map sending $T$ to $\sfread(T)$ is a poset isomorphism from $(E,\preceq_E)$ to $([\sfread(T_E),\sfread(T'_E)]_L,\preceq_L)$ (\cref{prop: class is weak interval}).
Moreover, \cref{thm: component interval module} shows that this poset isomorphism extends linearly to an $H_m(0)$-module isomorphism
\[
\scrG_E\cong\sfB(\sfread(T_E),\sfread(T'_E)).
\]
Consequently, we have an $H_m(0)$-module isomorphism
\[
\scrG_{\lambda;m}
=\bigoplus_{E\in\calE_{\lambda,m}}\scrG_E
\cong
\bigoplus_{E\in\calE_{\lambda,m}}
\sfB(\sfread(T_E),\sfread(T'_E)).
\]

\section{Preliminaries}\label{Sec: preliminaries}

For integers $a$ and $b$, let $[a,b]:=\{k\in \Z\mid a\le k\le b\}$ if $a\le b$, and let $[a,b]:=\emptyset$ otherwise. 
For $n \in \Z_{\ge0}$, we write $[n]$ for $[1,n]$.
Throughout the paper, all vector spaces are over $\C$.

\subsection{Compositions, partitions, and Young diagrams}
\label{subsec: comp and diag prelim}

In this section, let $n \in \Z_{\ge 0}$.
A \emph{composition} $\alpha$ of $n$, denoted by $\alpha\models n$, is a finite sequence of positive integers
\[
\alpha=(\alpha_1,\alpha_2,\ldots,\alpha_k)
\]
such that $\alpha_1+\alpha_2+\cdots+\alpha_k=n$. We write
$|\alpha|=n$ and $\ell(\alpha)=k$. The empty composition $\varnothing$ is the unique composition of $0$.

For $n \ge 1$ and $\alpha=(\alpha_1,\alpha_2,\ldots,\alpha_k)\models n$, define
\[
\set(\alpha):=
\{\alpha_1+\alpha_2+\cdots+\alpha_i\mid 1\le i\le k-1\}
\subseteq[n-1].
\]
Conversely, for a nonempty subset
$I=\{i_1<i_2<\cdots<i_\ell\}\subseteq [n-1]$, define
\[
\comp(I) := (i_1,i_2-i_1,\ldots,i_\ell-i_{\ell-1},n-i_\ell),
\]
and set $\comp(\emptyset):=(n)$.
The maps $\set$ and $\comp$ are mutually inverse bijections.
For $n=0$, we adopt the conventions $\set(\varnothing):=\emptyset$ and $\comp(\emptyset):=\varnothing$.

For $\lambda=(\lambda_1,\lambda_2,\ldots,\lambda_{\ell(\lambda)})\models n$, if
\[
\lambda_1\ge \lambda_2\ge\cdots\ge \lambda_{\ell(\lambda)},
\]
then we say that $\lambda$ is a \emph{partition} of $n$ and denote it by $\lambda\vdash n$.
For $\lambda=(\lambda_1,\lambda_2,\ldots,\lambda_{\ell(\lambda)})\vdash n$, we define the \emph{Young diagram} $\tyd(\lambda)$ of $\lambda$ by a left-justified array of $n$ boxes where the $i$th row from the top has $\lambda_i$ boxes for $1 \le i \le \ell(\lambda)$. 
We write $(i,j)$ for the box in the $i$th row and $j$th column. 
Thus $(i,j)\in\tyd(\lambda)$ means that $1\le i\le \ell(\lambda)$ and $1\le j\le \lambda_i$.

\subsection{Symmetric functions and quasisymmetric functions}
\label{subsec: qsym prelim}

A \emph{symmetric function} is a bounded-degree formal power series in the variables $x_1,x_2,x_3,\ldots$ with coefficients in $\Z$ that is invariant under every permutation of the variables. 
The ring $\Sym$ of symmetric functions is a graded $\Z$-algebra
\[
\Sym=\bigoplus_{n\ge 0}\Sym_n,
\]
where $\Sym_n$ is the $\Z$-module consisting of all symmetric functions that are homogeneous of degree $n$.

A \emph{semistandard Young tableau} of shape $\lambda$ is a filling of $\tyd(\lambda)$ with positive integers such that the entries weakly increase along rows and strictly increase along columns. 
Let $\SSYT(\lambda)$ denote the set of semistandard Young tableaux of shape $\lambda$.
For $T\in\SSYT(\lambda)$, let $\wt(T)=(a_1,a_2,\ldots)$, where $a_i$ is the number
of entries equal to $i$ in $T$, and let $\bfx^{\wt(T)}=x_1^{a_1}x_2^{a_2}\cdots$. 
The \emph{Schur function} indexed by $\lambda$ is
\[
s_\lambda=\sum_{T\in\SSYT(\lambda)}\bfx^{\wt(T)}.
\]
For every $n\ge 0$, the set $\{s_\lambda\mid \lambda\vdash n\}$ is a basis for $\Sym_n$. Thus the Schur functions form a basis for $\Sym$, called the \emph{Schur basis}.

\emph{Quasisymmetric functions} are bounded-degree formal power series in the variables $x_1,x_2,x_3,\ldots$ with coefficients in $\Z$ such that, for every composition $\alpha=(\alpha_1,\alpha_2,\ldots,\alpha_k)$, the coefficient of $x_{i_1}^{\alpha_1}x_{i_2}^{\alpha_2}\cdots x_{i_k}^{\alpha_k}$ is independent of the choice of $i_1<i_2<\cdots<i_k$. 
The ring $\QSym$ of quasisymmetric functions is a graded $\Z$-algebra
\[
\QSym=\bigoplus_{n\ge 0}\QSym_n,
\]
where $\QSym_n$ is the $\Z$-module consisting of all quasisymmetric functions that are homogeneous of degree $n$. 
We regard $\Sym$ as a subalgebra of $\QSym$.

For $\alpha\models n$, the \emph{fundamental quasisymmetric function} $F_\alpha$ is defined by
\[
F_\alpha
=\sum_{\substack{1\le i_1\le i_2\le \cdots\le i_n\\
i_j<i_{j+1}\text{ if }j\in \set(\alpha)}}
x_{i_1}x_{i_2}\cdots x_{i_n},
\]
with the convention $F_\varnothing=1$.
For every $n\ge 0$, the set $\{F_\alpha\mid \alpha\models n\}$ is a basis for $\QSym_n$.
We call an expansion in the basis of fundamental quasisymmetric functions an \emph{$F$-expansion}.

A \emph{standard Young tableau} of shape $\lambda\vdash n$ is a filling of $\tyd(\lambda)$ with the entries $1,2,\ldots,n$, each appearing once, such that the entries increase along rows and columns. Let $\SYT(\lambda)$ denote the set of standard Young tableaux of shape $\lambda$. For $T\in\SYT(\lambda)$, define
\[
\Des(T):=\{i\in[n-1]\mid i+1\text{ lies in a lower row than }i\}.
\]
It is well known that Schur functions are positive in the fundamental basis; more precisely,
\begin{align}\label{eq: F expansion of Schur}
s_\lambda=\sum_{T\in\SYT(\lambda)}F_{\comp(\Des(T))}.
\end{align}
For more information on symmetric functions and quasisymmetric functions, see \cite[Chapter 7]{Stanley99}.

\subsection{Grothendieck polynomials}
\label{subsec: groth prelim}

Grothendieck polynomials were originally defined by Lascoux and Sch\"utzenberger using divided difference operators \cite{LS82}.
Fomin and Kirillov~\cite{FK94} introduced a $\beta$-deformation of Grothendieck polynomials and studied the corresponding stable limits, now called stable $\beta$-Grothendieck polynomials.
At $\beta=-1$, Buch provided a set-valued tableau formula for stable Grothendieck polynomials \cite{Buch02}.
Buch's description extends to stable $\beta$-Grothendieck polynomials for general $\beta$.
We follow this combinatorial tableau description, which is used by various authors \cite{MPS21,MPPS20,PS19}.

A \emph{set-valued filling} of shape $\lambda$ assigns to each box of $\tyd(\lambda)$ a nonempty finite subset of $\Z_{>0}$. 
A set-valued filling $T$ is a \emph{semistandard set-valued tableau} if
\begin{enumerate}[label = -, leftmargin = 4ex]
\item
$\max \, T(i,j) \le \min \, T(i,j+1)$ for all $(i,j) \in \tyd(\lambda)$ such that $(i,j+1) \in \tyd(\lambda)$, and
\item
$\max \, T(i,j)< \min \, T(i+1,j)$ for all $(i,j) \in \tyd(\lambda)$ such that $(i+1,j) \in \tyd(\lambda)$.
\end{enumerate}
Let $\SVT(\lambda)$ be the set of semistandard set-valued tableaux of shape $\lambda$.
For $T \in \SVT(\lambda)$, we denote by $|T|$ the number of entries of $T$.
A semistandard set-valued tableau $T$ is called a \emph{standard set-valued tableau} if the entries of $T$ are exactly $\{1,2,\ldots, |T| \}$, each appearing exactly once.
Let $\StdSVT(\lambda)$ denote the set of standard set-valued tableaux of shape $\lambda$.
For $m \ge |\lambda|$, let $\StdSVT_m(\lambda):=\{T\in\StdSVT(\lambda)\mid |T|=m\}$.

For $S\in \SVT(\lambda)$, let $\wt(S)=(a_1,a_2,\ldots)$, where $a_i$ is the number of occurrences of $i$ among all entries of $S$, and define
\[
\bfx^{\wt(S)}:=x_1^{a_1}x_2^{a_2}\cdots .
\]
For a partition $\lambda$, the \emph{stable $\beta$-Grothendieck polynomial} is defined by
\[
G_\lambda^{(\beta)}
:=\sum_{S\in \SVT(\lambda)}
\beta^{|S|-|\lambda|} \bfx^{\wt(S)}.
\]
For $m\ge |\lambda|$, we denote by $G_{\lambda,m}^{(\beta)}$ the homogeneous degree $m$ component of $G_\lambda^{(\beta)}$.
In this paper, we set $\beta=1$ and write $G_\lambda:=G_\lambda^{(1)}$ and $G_{\lambda,m}:=G_{\lambda,m}^{(1)}$.

\begin{remark}
Let $m \ge |\lambda|$.
Taking the homogeneous degree $m$ component of the Schur expansion in
\cite[Corollary~3.11]{MPS21}, we have
\[
G_{\lambda,m}^{(\beta)} = \sum_{\mu \vdash m} \beta^{|\mu|-|\lambda|} M_{\lambda}^\mu s_\mu,
\]
where $M_{\lambda}^\mu$ denotes the number of Yamanouchi set-valued tableaux of shape $\lambda$ and weight $\mu$.
It follows that the $\beta=1$ specialization $G_{\lambda,m}$ is positive in the Schur basis.
By the expansion \cref{eq: F expansion of Schur}, $G_{\lambda,m}$ is also positive in the fundamental basis.
This is the reason we work with the $\beta=1$ specialization in this paper.
In contrast, 
\[
G_{\lambda,m}^{(-1)} = \sum_{\mu \vdash m} (-1)^{|\mu|-|\lambda|} M_{\lambda}^\mu s_\mu.
\]
If $m-|\lambda|$ is odd, this component is the negative of a nonzero $F$-positive function and therefore cannot be the quasisymmetric characteristic of an $H_m(0)$-module.
\end{remark}

\subsection[The 0-Hecke algebra and the quasisymmetric characteristic]
{The \texorpdfstring{$0$}{0}-Hecke algebra and the quasisymmetric characteristic}
\label{subsec: 0hecke prelim}

The symmetric group $\SG_n$ is generated by the simple transpositions $s_i=(i,i+1)$ for $1\le i\le n-1$. 
For $\sigma \in \SG_n$, the number of simple transpositions in a reduced expression for $\sigma$ is called the \emph{length} of $\sigma$ and is denoted by $\ell(\sigma)$.

The $0$-Hecke algebra $H_n(0)$ is the unital $\C$-algebra generated by $\pi_1,\pi_2,\ldots,\pi_{n-1}$ subject to the relations
\begin{alignat}{2}\label{eq: relations of pi}
\pi_i^2 &= \pi_i 
&\qquad& \text{for } 1\le i \le n-1, \nonumber \\ 
\pi_i\pi_{i+1}\pi_i &= \pi_{i+1}\pi_i\pi_{i+1}
&\qquad& \text{for } 1\le i \le n-2, \\ \nonumber
\pi_i\pi_j &= \pi_j\pi_i
&\qquad& \text{if } |i-j|>1.
\end{alignat}
For $\sigma \in \SG_n$, let $\pi_\sigma:=\pi_{i_1}\pi_{i_2}\cdots\pi_{i_k}$, where $s_{i_1}s_{i_2}\cdots s_{i_k}$ is a reduced expression for $\sigma$.
It is well known that $\pi_\sigma$ is independent of the choices of reduced expressions and that $\{\pi_\sigma \mid \sigma \in \SG_n\}$ is a $\C$-basis for $H_n(0)$.
By \cite{Norton79}, there are exactly $2^{n-1}$ irreducible $H_n(0)$-modules. They are naturally parametrized by compositions of $n$.
For $\alpha\models n$, let $\bfF_\alpha$ be the one-dimensional $H_n(0)$-module spanned by $v_\alpha$ with action
\[
\pi_i\cdot v_\alpha=
\begin{cases}
0, & i\in \set(\alpha),\\
v_\alpha, & i\notin \set(\alpha).
\end{cases}
\]

Let $G_0(\Hmod)$ be the Grothendieck group of the category of finite-dimensional $H_n(0)$-modules. In~\cite{DKLT96}, Duchamp, Krob, Leclerc, and Thibon introduced the \emph{quasisymmetric characteristic}
\[
\ch:\bigoplus_{n\ge 0}G_0(\Hmod) \rightarrow \QSym,
\quad
[\bfF_\alpha]\mapsto F_\alpha.
\]

\subsection{Weak Bruhat interval modules}
\label{subsec: weak bruhat prelim}

Given $\sigma \in \SG_n$ and $i \in [n-1]$, the integer $i$ is called a \emph{left descent} of $\sigma$ if $\ell(s_i\sigma)<\ell(\sigma)$.
Let $\DesL(\sigma)$ be the set of all left descents of $\sigma$.
The \emph{left weak Bruhat order} $\preceq_L$ on $\SG_n$ is the partial order whose covering relation $\preceq_L^c$ is defined by
\[
\sigma \preceq_L^c s_i\sigma
\quad\text{if and only if}\quad
i \notin \DesL(\sigma).
\]
Given $\sigma,\rho \in \SG_n$, the closed interval
\[
[\sigma,\rho]_L:=\{\gamma \in \SG_n\mid \sigma \preceq_L \gamma \preceq_L \rho\}
\]
is called the \emph{left weak Bruhat interval} from $\sigma$ to $\rho$.

\begin{definition}\label{def: weak bruhat interval module}(\cite[Definition 1]{JKLO22})
Let $\sigma,\rho \in \SG_n$.
The \emph{weak Bruhat interval module associated to $[\sigma,\rho]_L$}, denoted by $\sfB(\sigma,\rho)$, is the $H_n(0)$-module with underlying space $\C[\sigma,\rho]_L$ and with the $H_n(0)$-action defined by
\[
\pi_i\cdot \gamma:=
\begin{cases}
\gamma & \text{if $i \in \DesL(\gamma)$},\\
0 & \text{if $i \notin \DesL(\gamma)$ and $s_i\gamma \notin [\sigma,\rho]_L$},\\
s_i\gamma & \text{if $i \notin \DesL(\gamma)$ and $s_i\gamma \in [\sigma,\rho]_L$.}
\end{cases}
\]
\end{definition}

\section[0-Hecke modules arising from standard set-valued tableaux]
{$0$-Hecke modules arising from standard set-valued tableaux}
\label{sec: svt action}

In this section, we construct an $H_m(0)$-module whose quasisymmetric characteristic is $G_{\lambda,m}$.
We first give an $F$-expansion of $G_{\lambda,m}$ in terms of standard set-valued tableaux.
We then define an $H_m(0)$-action on $\C\StdSVT_m(\lambda)$ and, using the $F$-expansion, show that the resulting module has $G_{\lambda,m}$ as its quasisymmetric characteristic.
Finally, we provide a direct sum decomposition of this module into submodules, each of which will be shown to be a weak Bruhat interval module in \cref{sec: WBIM decomposition}.

From now on, let $n$ be a positive integer and let $\lambda = (\lambda_1, \lambda_2, \ldots, \lambda_l)$ be a partition of $n$, unless otherwise stated.

\subsection[An F-expansion via standard set-valued tableaux]{An $F$-expansion of $G_{\lambda,m}$ via standard set-valued tableaux}

Following \cite[Definition~3.1]{BM12} and \cite[Definition~4.1]{PPPS22}, where French notation is used, for $T \in \SVT(\lambda)$, define the reading word $w_T$ as follows.

\begin{enumerate}[label = {\rm (R\arabic*)}]
\item Read the rows from bottom to top.
\item Within each row, first read the non-minimal entries. Read the boxes from
right to left and the entries within each box in decreasing order.
\item Then read the minimal entries from left to right.
\end{enumerate}

We define the standardization of set-valued tableaux by the following algorithm.

\begin{algorithm}\label{alg:standardization}
Let $T \in \SVT(\lambda)$ and let
\[
c_1<c_2<\cdots<c_r
\]
be the distinct integers appearing in $T$.
\begin{enumerate}[label = {\it Step \arabic*.}]
\item Let $S$ be a copy of $T$.
Set $a=1$ and $p = 1$.
\item Let $x^{(p)}_1, x^{(p)}_2,\ldots, x^{(p)}_{k_p}$ be the occurrences of $c_p$ in the original filling $T$, listed from left to right in the word $w_T$.
In $S$, replace the occurrence corresponding to $x^{(p)}_i$ by $a+i-1$ for $1\le i \le k_p$.
\item If $p<r$, set $a = a + k_p$ and $p = p+1$, and go to {\it Step 2}.
Otherwise go to {\it Step 4}.
\item Define $\std(T) := S$ and terminate the algorithm.
\end{enumerate}
\end{algorithm}

\begin{example}
Let
\[\ytableausetup{boxsize=2em}
T=
\begin{array}{l}
\begin{ytableau}
\scriptstyle 1,2 & \scriptstyle 2 & \scriptstyle 2,3,4 & \scriptstyle 4,5,6\\
\scriptstyle 5,6
\end{ytableau}
\end{array}
\in \SVT((4,1)).
\]
Then
\[
w_T=6\,5\,6\,5\,4\,3\,2\,1\,2\,2\,4.
\]
The integer $1$ is replaced by $1$, the three occurrences of $2$ are replaced
by $2,3,4$, the integer $3$ is replaced by $5$, the two occurrences of $4$ are
replaced by $6,7$, the two occurrences of $5$ are replaced by $8,9$, and the
two occurrences of $6$ are replaced by $10,11$.
Hence
\[
\std(T)=
\begin{array}{l}
\begin{ytableau}
\scriptstyle 1,2 & \scriptstyle 3 & \scriptstyle 4,5,6 & \scriptstyle 7,9,11\\
\scriptstyle 8,10
\end{ytableau}
\end{array}.
\]
\end{example}

\begin{lemma}\label{lem:standardization-is-standard}
Let $T\in\SVT(\lambda)$ and $m = |T|$. 
Then $\std(T)\in\StdSVT_m(\lambda)$.
\end{lemma}

\begin{proof}
The algorithm assigns the labels $1,2,\ldots,m$ exactly once.
It remains to show that the resulting filling is a set-valued tableau.

Let $a$ and $b$ be two entries of $T$ with $a < b$.
Then all occurrences of $a$ are replaced before all occurrences of $b$ in the algorithm.
Thus, every label assigned to an occurrence of $a$ is smaller than every label assigned to an occurrence of $b$.

Let $x_1$ and $x_2$ be two occurrences of the same integer in $T$.
Then $x_1$ and $x_2$ cannot lie in the same box or in the same column.
Suppose that $x_1$ lies in the box immediately to the left of the box containing $x_2$.
Let $B_i$ be the box containing $x_i$ for $i=1,2$.
The row condition for $T$ gives $x_1=\max B_1=\min B_2=x_2$.
If $x_1$ is not minimal in $B_1$, then $x_1$ is read among the non-minimal entries, before the minimal entry $x_2$.
If $x_1$ is minimal in $B_1$, then both $x_1$ and $x_2$ are minimal entries, and $x_1$ is read before $x_2$ since minimal entries are read from left to right.
Since \cref{alg:standardization} labels equal integers in reading order, $x_1$ receives a smaller label than $x_2$.

Combining the two cases, we see that the row and column inequalities are preserved.
Thus $\std(T)$ is a standard set-valued tableau.
\end{proof}

For $m \ge |\lambda|$, $T \in \StdSVT_m(\lambda)$, and $i \in [m]$, let $\row_T(i)$ and $\col_T(i)$ denote the row and column indices of the box containing $i$, respectively.
Define
\[
\Des(T):=\{i \in [m-1] \mid \col_T(i) \ge \col_T(i+1)\}.
\]

\begin{lemma}\label{lem:descent box criterion}
For $m\ge |\lambda|$, let $T\in\StdSVT_m(\lambda)$ and $i\in[m-1]$.
Then $i\in\Des(T)$ if and only if $i+1$ appears before $i$ in $w_T$.
\end{lemma}

\begin{proof}
We first prove the ``if'' part.
Suppose that $i+1$ appears before $i$ in $w_T$.
Then, by the construction of $w_T$, we have $\row_T(i)\le \row_T(i+1)$.
If $\row_T(i) < \row_T(i+1)$, then $\col_T(i)\ge \col_T(i+1)$ by the row and column inequalities for set-valued tableaux.
Hence $i \in \Des(T)$ in this case.
It remains to consider the case $\row_T(i)=\row_T(i+1)$.
If $\col_T(i)=\col_T(i+1)$, then $i \in \Des(T)$.
Suppose that $\col_T(i)\ne \col_T(i+1)$.
Then $\col_T(i)<\col_T(i+1)$ by the row inequalities, and $i+1$ is minimal in its box.
If $i$ is not minimal in its box, then $i$ is read before all minimal entries in the row.
If $i$ is minimal in its box, then $i$ is read before $i+1$ since minimal entries are read from left to right.
Both cases contradict the assumption.
Thus $\col_T(i) \ge \col_T(i+1)$, and $i \in \Des(T)$.

We next prove the ``only if'' part.
Suppose that $i \in \Des(T)$.
If $\row_T(i)>\row_T(i+1)$, then $\col_T(i)<\col_T(i+1)$ by the row and column inequalities, a contradiction.
Thus $\row_T(i)\le \row_T(i+1)$.
If $\row_T(i)<\row_T(i+1)$, then the row containing $i+1$ is read before the row containing $i$.
Hence, $i+1$ appears before $i$ in $w_T$.
It remains to consider the case $\row_T(i)=\row_T(i+1)$.
Since $i \in \Des(T)$, we have $\col_T(i)\ge \col_T(i+1)$.
By the row inequalities, $\col_T(i)=\col_T(i+1)$, so $i$ and $i+1$ lie in the same box.
By the definition of the reading word, $i+1$ is read before $i$ within this box.
Hence, $i+1$ appears before $i$ in $w_T$.
\end{proof}

\begin{example}\label{ex: descent reading example}
Let
\[\ytableausetup{boxsize=2em}
T=
\begin{array}{l }
\begin{ytableau}
\scriptstyle 1,2 & \scriptstyle 3 & \scriptstyle 4,5,6 & \scriptstyle 7,9,11\\
\scriptstyle 8,10
\end{ytableau}  
\end{array}
\in \StdSVT_{11}((4,1)).
\]
We see that
\[
\Des(T)=\{1,4,5,7,9\}.
\]
By \cref{lem:descent box criterion}, the same descent set is obtained from the reading word.
Indeed, the bottom row contributes $10 \, 8$.
In the top row, the non-minimal entries $11$, $9$, $6$, $5$, and $2$ are read first, followed by the minimal entries $1$, $3$, $4$, and $7$.
Hence
\[
w_T=10\,8\,11\,9\,6\,5\,2\,1\,3\,4\,7.
\]
In particular, $i+1$ appears before $i$ in $w_T$ for $i=1,4,5,7,9$, and after $i$ for $i=2,3,6,8,10$.
\end{example}

The following theorem gives a fundamental expansion of the homogeneous component $G_{\lambda,m}$ in terms of standard set-valued tableaux.
This expansion can also be deduced from the Hecke insertion results of Patrias and Pylyavskyy \cite[Theorems~6.17 and 6.19]{PP16} and the relation between stable and weak stable Grothendieck functions in \cite[Proposition~9.22]{LP07}.
This deduction requires translating the tableau convention in \cite[Remark~6.10]{PP16} and complementing descent sets.
To avoid introducing several notions that are not otherwise needed, we include a direct proof by standardization.

\begin{theorem}\label{thm: stable groth F expansion}
For $\lambda \vdash n$ and $m\ge n$,
\begin{equation}\label{eq: stable groth F expansion}
G_{\lambda,m}
=\sum_{T\in \StdSVT_m(\lambda)}
F_{\comp(\Des(T))}.
\end{equation}
\end{theorem}

\begin{proof}
For each $S \in \SVT(\lambda)$ with $|S|=m$, $\std(S)\in\StdSVT_m(\lambda)$ by \cref{lem:standardization-is-standard}.
It follows that
\[
G_{\lambda,m}
=\sum_{\substack{S \in \SVT(\lambda) \\ |S|=m}}\bfx^{\wt(S)}
=\sum_{T \in \StdSVT_m(\lambda)} \sum_{\substack{S \in \SVT(\lambda) \\ \std(S) = T}}
\bfx^{\wt(S)}.
\]
Thus it suffices to show that, for each $T \in \StdSVT_m(\lambda)$,
\begin{align}\label{eq: stable groth F expansion for a fixed T}
\sum_{\substack{S \in \SVT(\lambda) \\ \std(S)=T}} \bfx^{\wt(S)}
= \sum_{\substack{1 \le i_1 \le i_2 \le \cdots \le i_m \\ i_j < i_{j+1} \text{ if } j \in \Des(T)}} x_{i_1}x_{i_2}\cdots x_{i_m}.
\end{align}

Let $T\in\StdSVT_m(\lambda)$.
Choose a sequence $\mathbf{i}=(i_1,i_2,\ldots,i_m)$ of integers satisfying
\begin{align}\label{eq: an increasing seq for T}
1 \le i_1\le i_2\le\cdots\le i_m
\quad\text{and}\quad
i_j<i_{j+1}\quad\text{for all }j\in\Des(T).
\end{align}
Let $S_{T,\mathbf{i}}$ be the filling obtained from $T$ by replacing each entry $j$ of $T$ with $i_j$.
We claim that $S_{T,\mathbf{i}}$ is a semistandard set-valued tableau and that
$\std(S_{T,\mathbf{i}})=T$.

We first prove that $S_{T,\mathbf{i}}$ is a semistandard set-valued tableau.
We need to show row inequalities, column inequalities, and distinctness inside each box.
Let $p<q$ be two labels of $T$.
Then $i_p\le i_q$.
If $\row_T(p)=\row_T(q)$ and $\col_T(p)+1=\col_T(q)$, then the row inequality follows from $i_p\le i_q$.
Suppose that $\row_T(q)=\row_T(p)+1$ and $\col_T(p)=\col_T(q)$.
Then there exists at least one integer in $[p,q-1] \cap \Des(T)$; otherwise,
\[
\col_T(p)<\col_T(p+1)<\cdots<\col_T(q),
\]
which contradicts $\col_T(p)=\col_T(q)$.
Since the sequence in \cref{eq: an increasing seq for T} is chosen to be strict at the descents of $T$, we have $i_p < i_q$.
This proves the column inequalities.
If $\row_T(p)=\row_T(q)$ and $\col_T(p)=\col_T(q)$, then $[p,q-1]\cap\Des(T)\ne\varnothing$ by the argument above, and hence $i_p<i_q$.
This shows that no box contains the same integer twice.
Therefore $S_{T,\mathbf{i}}$ is a semistandard set-valued tableau.

We next prove that $\std(S_{T,\mathbf{i}})=T$.
Replacing each entry $j$ of $T$ by $i_j$ preserves the reading order of the corresponding entries.
Let $[u,v]$ be a maximal interval of indices on which $\mathbf{i}$ is constant, and denote the common value by $a$.
By \cref{eq: an increasing seq for T}, we have $r\notin\Des(T)$ for all $u \le r < v$.
By \cref{lem:descent box criterion}, $r$ appears before $r+1$ in $w_T$ for every $u \le r < v$.
Since the entries $u,u+1,\ldots,v$ of $T$ are all replaced by $a$ and their reading order is preserved, \cref{alg:standardization} assigns the labels $u,u+1,\ldots,v$ to the occurrences of $a$.
Applying this argument to every maximal interval on which $\mathbf{i}$ is constant gives $\std(S_{T,\mathbf{i}})=T$.

Conversely, let $S$ be a semistandard set-valued tableau with $\std(S)=T$.
Let $k_j$ be the integer in $S$ that is replaced by $j$ when $S$ is standardized to $T$.
Then we have a sequence $\mathbf{k}=(k_1,k_2,\ldots,k_m)$ of integers satisfying
\[
1\le k_1\le k_2\le\cdots\le k_m.
\]
Let $j\in\Des(T)$. 
Then $j+1$ appears before $j$ in $w_T$ by \cref{lem:descent box criterion}.
Suppose that $k_j=k_{j+1}$.
By \cref{alg:standardization}, the occurrence that is replaced by $j$ appears before the occurrence that is replaced by $j+1$ in $w_S$.
Since corresponding entries appear in the same order in $w_S$ and $w_T$, $j$ appears before $j+1$ in $w_T$.
This contradicts the fact that $j+1$ appears before $j$ in $w_T$.
It follows that $k_j<k_{j+1}$ for all $j\in\Des(T)$, so $\mathbf{k}$ satisfies \cref{eq: an increasing seq for T}.
By the definition of $\mathbf{k}$, we also have $S=S_{T,\mathbf{k}}$.

Consequently, the equality in \cref{eq: stable groth F expansion for a fixed T} holds, as desired.
\end{proof}

\subsection[A 0-Hecke action on standard set-valued tableaux]{An $H_m(0)$-action on $\StdSVT_m(\lambda)$}

Hereafter, let $m$ be an integer with $m \ge n$.
For $1\le i\le m-1$, let $s_i\cdot T$ be the set-valued filling obtained from $T$ by interchanging the entries $i$ and $i+1$.
For each $1\le i\le m-1$, define a linear operator $\uppi_i:\C \StdSVT_m(\lambda) \rightarrow \C \StdSVT_m(\lambda)$ by letting
\begin{align}\label{eq: svt action}
\uppi_i(T):=
\begin{cases}
T & \text{if $\col_T(i)<\col_T(i+1)$},\\
0 & \text{if $\col_T(i)=\col_T(i+1)$},\\
s_i \cdot T & \text{if $\col_T(i)>\col_T(i+1)$}
\end{cases}
\end{align}
for $T\in\StdSVT_m(\lambda)$ and extending by linearity.

The following theorem is the first main result of this subsection.

\begin{theorem}\label{thm: svt action is 0hecke}
Let $\lambda \vdash n$ and $m\ge n$.
The operators $\uppi_1, \uppi_2, \ldots, \uppi_{m-1}$ on $\C \StdSVT_m(\lambda)$ define an $H_m(0)$-action on $\C \StdSVT_m(\lambda)$.
\end{theorem}

In order to prove this theorem, let us establish some necessary lemmas.

\begin{lemma}\label{lem: admissible swap}
Let $T\in\StdSVT_m(\lambda)$ and $1 \le i \le m-1$. 
If $\col_T(i)>\col_T(i+1)$, then $s_i\cdot T\in\StdSVT_m(\lambda)$.
\end{lemma}

\begin{proof}
Assume that $\col_T(i)>\col_T(i+1)$.
By the row and column inequalities for set-valued tableaux, $\row_T(i)<\row_T(i+1)$.
This implies that switching $i$ and $i+1$ in $T$ does not violate the row and column inequalities.
Therefore $s_i\cdot T\in\StdSVT_m(\lambda)$.
\end{proof}

\begin{lemma}\label{lem: svt action idempotence}
For $1\le i\le m-1$, $\uppi_i^2=\uppi_i$.
\end{lemma}

\begin{proof}
Let $T\in\StdSVT_m(\lambda)$.
If $\uppi_i(T)=T$ or $\uppi_i(T)=0$, then the assertion is clear.
Suppose that $\uppi_i(T)=s_i\cdot T$.
Then we have $\col_T(i) > \col_T(i+1)$, equivalently $\col_{s_i\cdot T}(i)<\col_{s_i\cdot T}(i+1)$.
It follows that $\uppi_i(s_i\cdot T)=s_i\cdot T$, and hence $\uppi_i^2(T)=\uppi_i(T)$.
\end{proof}

\begin{lemma}\label{lem: svt action commutation}
For $1\le i,j\le m-1$ with $|i-j|>1$, $\uppi_i\uppi_j=\uppi_j\uppi_i$.
\end{lemma}

\begin{proof}
Let $T\in\StdSVT_m(\lambda)$.
Suppose that $\uppi_i(T)=T$ or $\uppi_i(T)=0$.
If $\uppi_j(T)=T$ or $\uppi_j(T)=0$, then it is clear that $\uppi_i\uppi_j(T) = \uppi_j\uppi_i(T)$.
Suppose that $\uppi_j(T)=s_j\cdot T$.
Since $|i-j|>1$, applying $s_j$ does not change the columns containing $i$ and $i+1$.
Thus
\[
\col_T(i)=\col_{s_j\cdot T}(i)
\quad\text{and}\quad
\col_T(i+1)=\col_{s_j\cdot T}(i+1).
\]
Combining this with the assumption that $\uppi_i(T)=T$ or $\uppi_i(T)=0$, we have
\[
\uppi_i\uppi_j(T)=\uppi_j\uppi_i(T).
\]

By symmetry, it remains to consider the case where $\uppi_i(T)=s_i\cdot T$ and $\uppi_j(T)=s_j\cdot T$.
Since $|i-j|>1$, applying $s_j$ does not change the columns containing $i$ and $i+1$, and applying $s_i$ does not change the columns containing $j$ and $j+1$.
Thus
\[
\col_T(i)=\col_{s_j\cdot T}(i),
\quad
\col_T(i+1)=\col_{s_j\cdot T}(i+1),
\]
and
\[
\col_T(j)=\col_{s_i\cdot T}(j),
\quad
\col_T(j+1)=\col_{s_i\cdot T}(j+1).
\]
It follows that
\[
\uppi_i(s_j\cdot T)=s_i\cdot(s_j\cdot T)
\quad\text{and}\quad
\uppi_j(s_i\cdot T)=s_j\cdot(s_i\cdot T).
\]
Since $s_i\cdot(s_j\cdot T)=s_j\cdot(s_i\cdot T)$, we have
\[
\uppi_i \uppi_j (T)
= \uppi_i(s_j\cdot T)
= s_i\cdot(s_j\cdot T)
= s_j\cdot(s_i\cdot T)
= \uppi_j(s_i\cdot T)
= \uppi_j \uppi_i (T).
\]
This proves the assertion.
\end{proof}

\begin{lemma}\label{lem: svt action braid}
For $1\le i\le m-2$, $\uppi_i\uppi_{i+1}\uppi_i=\uppi_{i+1}\uppi_i\uppi_{i+1}$.
\end{lemma}

\begin{proof}
Let $T\in\StdSVT_m(\lambda)$ and $1\le i\le m-2$.
For $U\in\StdSVT_m(\lambda)$, write
\[
\bfc(U)=(\col_U(i),\col_U(i+1),\col_U(i+2)),
\]
and set $\bfc(T)=(a,b,c)$.
We consider three cases according to the value of $\uppi_i(T)$.

\medskip
\noindent
{\bf Case 1: $\uppi_i(T)=T$.}
In this case, $a<b$.
If $b<c$, then $\uppi_i$ and $\uppi_{i+1}$ fix $T$, and hence $\uppi_i\uppi_{i+1}\uppi_i(T) = T = \uppi_{i+1}\uppi_i\uppi_{i+1}(T)$.
If $b=c$, then $\uppi_{i+1}(T)=0$, so $\uppi_i\uppi_{i+1}\uppi_i(T) = 0 = \uppi_{i+1}\uppi_i\uppi_{i+1}(T)$.

Suppose that $c < b$.
Then $\uppi_{i+1}(T) = s_{i+1} \cdot T$ and $\bfc(\uppi_{i+1}(T))=(a,c,b)$.
If $a<c$, then $\uppi_i\uppi_{i+1}(T)=\uppi_{i+1}(T)$.
Since $c<b$, we also have $\uppi_{i+1}^2(T)=\uppi_{i+1}(T)$.
It follows that
\[
\uppi_i\uppi_{i+1}\uppi_i(T) = \uppi_i\uppi_{i+1}(T) = \uppi_{i+1}(T) = \uppi_{i+1}^2(T) = \uppi_{i+1}\uppi_i\uppi_{i+1}(T).
\]
If $a=c$, then $\uppi_i\uppi_{i+1}(T) = 0$, and hence
\[
\uppi_i\uppi_{i+1}\uppi_i(T) = \uppi_i\uppi_{i+1}(T) = 0 = \uppi_{i+1}\uppi_i\uppi_{i+1}(T).
\]
If $c < a$, then $\uppi_i\uppi_{i+1}(T) = s_i \cdot \uppi_{i+1}(T) = s_i \cdot (s_{i+1} \cdot T)$, so $\bfc(\uppi_i \uppi_{i+1}(T)) = (c,a,b)$.
Since $a<b$, we have $\uppi_{i+1} \uppi_i\uppi_{i+1}(T) = s_i \cdot (s_{i+1} \cdot T)$.
Therefore
\[
\uppi_i\uppi_{i+1}\uppi_i(T) = \uppi_i\uppi_{i+1}(T) = s_i \cdot (s_{i+1} \cdot T) = \uppi_{i+1}\uppi_i\uppi_{i+1}(T).
\]

\medskip
\noindent
{\bf Case 2: $\uppi_i(T)=0$.}
In this case, $a=b$.
If $b<c$, then $\uppi_{i+1}(T)=T$, and hence
\[
\uppi_i\uppi_{i+1}\uppi_i(T) = 0 = \uppi_{i+1}\uppi_i(T) = \uppi_{i+1}\uppi_i\uppi_{i+1}(T).
\]
If $b=c$, then $\uppi_i(T)=0=\uppi_{i+1}(T)$, so
\[
\uppi_i\uppi_{i+1}\uppi_i(T) = 0 = \uppi_{i+1}\uppi_i\uppi_{i+1}(T).
\]

Suppose that $c<b$.
Then $\uppi_{i+1}(T)=s_{i+1} \cdot T$ and $\bfc(\uppi_{i+1}(T))=(a,c,b)$.
Since $a=b>c$, we have $\uppi_i\uppi_{i+1}(T)=s_i \cdot (s_{i+1} \cdot T)$ and $\bfc(\uppi_i\uppi_{i+1}(T))=(c,a,b)$.
Since $a=b$, we have $\uppi_{i+1}\uppi_i\uppi_{i+1}(T)=0$.
Therefore
\[
\uppi_i\uppi_{i+1}\uppi_i(T) = 0 = \uppi_{i+1}\uppi_i\uppi_{i+1}(T).
\]

\medskip
\noindent
{\bf Case 3: $\uppi_i(T)=s_i\cdot T$.}
In this case, $b<a$.
We have $\bfc(\uppi_i(T))=(b,a,c)$.
Suppose that $b<c$.
Then $\uppi_{i+1}(T)=T$.
If $a<c$, then $\uppi_{i+1}\uppi_i(T)=\uppi_i(T)$.
Since $b<a$, we also have $\uppi_i^2(T)=\uppi_i(T)$.
It follows that
\[
\uppi_i\uppi_{i+1}\uppi_i(T) = \uppi_i^2(T) = \uppi_i(T) = \uppi_{i+1}\uppi_i(T) = \uppi_{i+1}\uppi_i\uppi_{i+1}(T).
\]
If $a=c$, then $\uppi_{i+1}\uppi_i(T)=0$, and hence
\[
\uppi_i\uppi_{i+1}\uppi_i(T) = 0 = \uppi_{i+1}\uppi_i(T) = \uppi_{i+1}\uppi_i\uppi_{i+1}(T).
\]
If $c<a$, then $\uppi_{i+1}\uppi_i(T)=s_{i+1} \cdot (s_i \cdot T)$ and $\bfc(\uppi_{i+1}\uppi_i(T))=(b,c,a)$.
Since $b<c$, we have $\uppi_i\uppi_{i+1}\uppi_i(T)=\uppi_{i+1} \uppi_i (T)$.
Therefore
\[
\uppi_i\uppi_{i+1}\uppi_i(T) = \uppi_{i+1}\uppi_i(T) = \uppi_{i+1}\uppi_i\uppi_{i+1}(T).
\]

Suppose that $b=c$.
Then $\uppi_{i+1}(T)=0$, so $\uppi_{i+1}\uppi_i\uppi_{i+1}(T)=0$.
Since $c<a$, we have $\uppi_{i+1}\uppi_i(T)=s_{i+1} \cdot (s_i \cdot T)$ and $\bfc(\uppi_{i+1}\uppi_i(T))=(b,c,a)$.
This together with the assumption $b=c$ implies that $\uppi_i\uppi_{i+1}\uppi_i(T)=0$.
Thus
\[
\uppi_i\uppi_{i+1}\uppi_i(T) = 0 = \uppi_{i+1}\uppi_i\uppi_{i+1}(T).
\]

Finally, suppose that $c<b$.
Since $c<b<a$, we have $\uppi_{i+1}\uppi_i(T)=s_{i+1} \cdot (s_i \cdot T)$ and $\bfc(\uppi_{i+1}\uppi_i(T))=(b,c,a)$.
This together with the assumption $c<b$ implies that 
\[
\uppi_i\uppi_{i+1}\uppi_i(T)=s_i \cdot (s_{i+1} \cdot (s_i \cdot T)).
\]
On the other hand, by the inequality $c<b<a$, we have $\uppi_{i+1}(T)=s_{i+1} \cdot T$ and $\bfc(\uppi_{i+1}(T))=(a,c,b)$.
Since $c < a$, we have $\uppi_i\uppi_{i+1}(T)=s_i \cdot (s_{i+1} \cdot T)$ and $\bfc(\uppi_i\uppi_{i+1}(T))=(c,a,b)$.
This together with the inequality $c<a$ implies that
\[
\uppi_{i+1}\uppi_i\uppi_{i+1}(T)=s_{i+1} \cdot (s_i \cdot (s_{i+1} \cdot T)).
\]
Therefore, $\uppi_i\uppi_{i+1}\uppi_i(T) 
= s_{i} \cdot (s_{i+1} \cdot (s_{i} \cdot T)) 
= s_{i+1} \cdot (s_i \cdot (s_{i+1} \cdot T)) 
= \uppi_{i+1}\uppi_i\uppi_{i+1}(T)$.
\end{proof}

Now, let us prove \cref{thm: svt action is 0hecke}.

\begin{proof}[Proof of \cref{thm: svt action is 0hecke}]
By \cref{lem: admissible swap}, the operators $\uppi_1, \uppi_2, \ldots,
\uppi_{m-1}$ in \cref{eq: svt action} are well-defined on
$\C\StdSVT_m(\lambda)$.
By \cref{lem: svt action idempotence}, \cref{lem: svt action commutation}, and \cref{lem: svt action braid}, they satisfy the relations in
\cref{eq: relations of pi}, and therefore define an $H_m(0)$-action on $\C\StdSVT_m(\lambda)$.
\end{proof}

By \cref{thm: svt action is 0hecke}, the vector space $\C\StdSVT_m(\lambda)$ is an $H_m(0)$-module.
We denote this module by $\scrG_{\lambda;m}$.

\begin{remark}
When $m=|\lambda|$, we have $\StdSVT_m(\lambda)=\SYT(\lambda)$.
In this case, the action in \cref{eq: svt action} agrees with Searles' convention \cite{Searles20}.
\end{remark}

The following theorem is the second main result of this subsection.

\begin{theorem}\label{thm: characteristic}
For $\lambda \vdash n$ and $m\ge n$, we have
\[
\ch([\scrG_{\lambda;m}])=G_{\lambda,m}.
\]
\end{theorem}

\begin{proof}
For $T\in\StdSVT_m(\lambda)$, define
\[
r(T):=\bigl|\{(p,q)\mid 1\le p<q\le m,\ \col_T(p)<\col_T(q)\}\bigr|.
\]
Set $N:=|\StdSVT_m(\lambda)|$, and enumerate the tableaux as $T_1,T_2,\ldots,T_N$ so that
\[
r(T_1)\ge r(T_2)\ge\cdots\ge r(T_N).
\]
Set $M_0:=\{0\}$ and for $1\le j\le N$, let
\[
M_j:=\C\{T_1,T_2,\ldots,T_j\}.
\]

Let us show that each $M_j$ is an $H_m(0)$-submodule of $\scrG_{\lambda;m}$.
Fix $1 \le j \le N$.
Let $1\le k\le j$ and $1\le i\le m-1$.
If $\pi_i \cdot T_k = T_k$ or $\pi_i \cdot T_k = 0$, then $\pi_i \cdot T_k \in M_j$.
Assume that $\pi_i \cdot T_k = s_i \cdot T_k$.
Then $\col_{T_k}(i)>\col_{T_k}(i+1)$.
In the sequence $(\col_{T_k}(1),\col_{T_k}(2),\ldots,\col_{T_k}(m))$, interchanging the terms in positions $i$ and $i+1$ creates the increasing pair $(i,i+1)$.
The total number of increasing pairs involving any other position remains unchanged.
Thus $r(s_i\cdot T_k)=r(T_k)+1$.
By the ordering of the tableaux, $s_i\cdot T_k=T_\ell$ for some $\ell<k$, and so $\pi_i\cdot T_k\in M_j$.
Hence $M_j$ is an $H_m(0)$-submodule of $\scrG_{\lambda;m}$.
Since $j$ was arbitrary,
\begin{align}\label{eq: filtration of svt module}
\{0\}=M_0\subseteq M_1\subseteq\cdots\subseteq M_N=\scrG_{\lambda;m}
\end{align}
is a filtration of $\scrG_{\lambda;m}$ by $H_m(0)$-submodules.

For $1\le j\le N$, the successive quotient $M_j/M_{j-1}$ of the filtration in \cref{eq: filtration of svt module} is one-dimensional with basis element $T_j + M_{j-1}$.
By the preceding argument and \cref{eq: svt action},
\[
\pi_i\cdot (T_j+M_{j-1}) =
\begin{cases}
T_j+M_{j-1} & \text{if $\col_{T_j}(i)<\col_{T_j}(i+1)$},\\
0 & \text{if $\col_{T_j}(i)\ge\col_{T_j}(i+1)$}.
\end{cases}
\]
It follows that $M_j/M_{j-1}$ is isomorphic to $\bfF_{\comp(\Des(T_j))}$.
Therefore,
\[
\ch([\scrG_{\lambda;m}])
=\sum_{j=1}^N\ch([M_j/M_{j-1}])
=\sum_{T\in\StdSVT_m(\lambda)}F_{\comp(\Des(T))} = G_{\lambda,m},
\]
where the last equality follows from \cref{thm: stable groth F expansion}.
\end{proof}

\begin{remark}\label{rem: indecomposable module}
As in the case of genomic Schur functions \cite[Remark~3.7]{KY24}, there are pairs $(\lambda,m)$ for which no indecomposable $H_m(0)$-module has $G_{\lambda,m}$ as its quasisymmetric characteristic.
For instance, consider the case $\lambda=(1^k)$, where $1 < k < m$.
Each tableau in $\StdSVT_m((1^k))$ is uniquely obtained by 
\begin{enumerate}[label = {\rm (\roman*)}, leftmargin = 4ex]
\item 
partitioning $[m]$ into $k$ nonempty parts, each consisting of consecutive integers, and 
\item 
placing one part in each box, in increasing order from top to bottom.
\end{enumerate}
Thus there are $\binom{m-1}{k-1}$ such tableaux.
Moreover, every $i \in [m-1]$ is a descent of every tableau in $\StdSVT_m((1^k))$, and hence
\[
G_{(1^k),m}=\binom{m-1}{k-1}F_{(1^m)}.
\]
On the other hand, by \cite[Theorem~4.7]{DHT02},
\[
\operatorname{Ext}^1_{H_m(0)}(\bfF_\alpha,\bfF_\alpha)=0
\]
for every $\alpha\models m$.
It follows that, for an $H_m(0)$-module $M$, if $\ch([M])=dF_\alpha$ for some $d \in \Z_{>0}$, then $M \cong \bfF_\alpha^{\oplus d}$.
Consequently, there is no indecomposable $H_m(0)$-module $M$ such that $\ch([M]) = G_{(1^k),m}$.
\end{remark}

\subsection[A direct sum decomposition]{A direct sum decomposition of $\scrG_{\lambda;m}$}
\label{subsec: decomposition}

The purpose of this subsection is to decompose $\scrG_{\lambda;m}$ into $H_m(0)$-submodules.
Let $T \in \StdSVT_m(\lambda)$.
For $1 \le j \le \lambda_1$, define
\[
c_j(T) := \left( |T(1,j)|,|T(2,j)|,\ldots,|T(h_j,j)| \right),
\]
where $h_j$ is the number of boxes in the $j$th column of $\tyd(\lambda)$.
Thus $c_j(T)$ records the number of entries in each box of the $j$th column from top to bottom.

\begin{definition}\label{def: column equivalence}
For $T,U \in \StdSVT_m(\lambda)$, define the equivalence relation $\sim_{\lambda,m}$ by
\[
T\sim_{\lambda,m}U
\quad\text{if and only if}\quad
c_j(T)=c_j(U)\text{ for all }1 \le j \le \lambda_1.
\]
\end{definition}

If $\lambda$ and $m$ are clear from the context, we write $\sim$ for $\sim_{\lambda,m}$.
Let $\calE_{\lambda,m}$ be the set of equivalence classes of $\StdSVT_m(\lambda)$ with respect to $\sim$.

\begin{example}\label{ex: column equivalence}
Let $\lambda=(2,1)$ and $m=4$.
Consider
\[
\begingroup
\ytableausetup{boxsize=1.7em}
T_1=
\begin{ytableau}
1,2 & 3\\
4
\end{ytableau}\;,
\qquad
T_2=
\begin{ytableau}
1,2 & 4\\
3
\end{ytableau}\;,
\qquad \text{and} \qquad
T_3=
\begin{ytableau}
1 & 2,3\\
4
\end{ytableau}\;.
\endgroup
\]
For $T_1$ and $T_2$, we have
\[
c_1(T_1)=c_1(T_2)=(2,1),
\quad \text{and} \quad
c_2(T_1)=c_2(T_2)=(1).
\]
Hence $T_1 \sim T_2$.
On the other hand,
\[
c_1(T_3)=(1,1),
\quad \text{and} \quad
c_2(T_3)=(2).
\]
Thus $T_3$ is equivalent to neither $T_1$ nor $T_2$.
\end{example}

\begin{theorem}\label{thm: component direct sum}
For each $E \in \calE_{\lambda,m}$, the $\C$-span of $E$ is closed under the $H_m(0)$-action.
\end{theorem}

\begin{proof}
Let $E \in \calE_{\lambda,m}$, $T \in E$, and $1 \le i \le m-1$.
In the cases where $\pi_i \cdot T=0$ or $\pi_i \cdot T=T$, it is clear that $\pi_i \cdot T$ is in the $\C$-span of $E$.
Suppose that $\pi_i \cdot T=s_i\cdot T$.
Then for any $1 \le j \le \lambda_1$,
\[
c_j(s_i \cdot T)=c_j(T)
\]
because interchanging $i$ and $i+1$ does not change the number of entries in each box.
It follows that $s_i \cdot T \sim T$, so $\pi_i \cdot T$ is in $E$.
Therefore the $\C$-span of $E$ is closed under the $H_m(0)$-action.
\end{proof}

For each $E \in \calE_{\lambda,m}$, we denote by $\scrG_E$ the $H_m(0)$-submodule of $\scrG_{\lambda;m}$ whose underlying space is the $\C$-span of $E$.
Since the equivalence classes in $\calE_{\lambda,m}$ form a partition of $\StdSVT_m(\lambda)$, \cref{thm: component direct sum} gives the following direct sum decomposition of $\scrG_{\lambda;m}$ into $H_m(0)$-submodules:
\[
\scrG_{\lambda;m}=\bigoplus_{E \in \calE_{\lambda,m}}\scrG_E.
\]

\begin{example}\label{ex: svt action 214}
We illustrate the $\pi_i$-actions on $\StdSVT_4((2,1))$ in \cref{fig: svt action 214}.
In this figure, the zero actions are omitted, and each column represents an equivalence class under $\sim_{(2,1),4}$.

\begin{figure}[ht]
\centering
\begingroup
\ytableausetup{boxsize=1.6em}
\begin{tikzpicture}[>=stealth,scale=0.76,transform shape,
  tableau/.style={inner sep=-2pt},
  edge/.style={->,shorten >=4pt,shorten <=4pt},
  loopedge/.style={->,out=45,in=315,loop,looseness=1,min distance=10mm}]
\node[tableau] (T1) at (0,0) {
\begin{ytableau}
1,2 & 3\\
4
\end{ytableau}};
\node[tableau] (T2) at (0,-2.2) {
\begin{ytableau}
1,2 & 4\\
3
\end{ytableau}};

\node[tableau] (T3) at (6.9,0) {
\begin{ytableau}
1 & 2,3\\
4
\end{ytableau}};
\node[tableau] (T4) at (6.9,-2.2) {
\begin{ytableau}
1 & 2,4\\
3
\end{ytableau}};
\node[tableau] (T6) at (6.9,-4.4) {
\begin{ytableau}
1 & 3,4\\
2
\end{ytableau}};

\node[tableau] (T5) at (13.8,0) {
\begin{ytableau}
1 & 2\\
3,4
\end{ytableau}};
\node[tableau] (T7) at (13.8,-2.2) {
\begin{ytableau}
1 & 3\\
2,4
\end{ytableau}};
\node[tableau] (T8) at (13.8,-4.4) {
\begin{ytableau}
1 & 4\\
2,3
\end{ytableau}};

\node at (3.45,-1.1) { $\bigoplus$};
\node at (10.35,-1.1) {$\bigoplus$};

\draw[edge] (T1) -- node[right] {$\pi_3$} (T2);
\draw[edge] (T3) -- node[right] {$\pi_3$} (T4);
\draw[edge] (T4) -- node[right] {$\pi_2$} (T6);
\draw[edge] (T5) -- node[right] {$\pi_2$} (T7);
\draw[edge] (T7) -- node[right] {$\pi_3$} (T8);

\node (L1) at ([xshift=0.175cm]T1.east) {};
\path (L1) edge [->,out=40,in=320,loop,min distance=10mm] (L1);
\node[anchor=west] at ([xshift=1.15cm]T1.east) {$\pi_2$};
\node (L2) at ([xshift=0.175cm]T2.east) {};
\path (L2) edge [->,out=40,in=320,loop,min distance=10mm] (L2);
\node[anchor=west] at ([xshift=1.15cm]T2.east) {$\pi_3$};
\node (L3) at ([xshift=0.175cm]T3.east) {};
\path (L3) edge [->,out=40,in=320,loop,min distance=10mm] (L3);
\node[anchor=west] at ([xshift=1.15cm]T3.east) {$\pi_1$};
\node (L4) at ([xshift=0.175cm]T4.east) {};
\path (L4) edge [->,out=40,in=320,loop,min distance=10mm] (L4);
\node[anchor=west] at ([xshift=1.15cm]T4.east) {$\pi_1,\pi_3$};
\node (L6) at ([xshift=0.175cm]T6.east) {};
\path (L6) edge [->,out=40,in=320,loop,min distance=10mm] (L6);
\node[anchor=west] at ([xshift=1.15cm]T6.east) {$\pi_2$};
\node (L5) at ([xshift=0.175cm]T5.east) {};
\path (L5) edge [->,out=40,in=320,loop,min distance=10mm] (L5);
\node[anchor=west] at ([xshift=1.15cm]T5.east) {$\pi_1$};
\node (L7) at ([xshift=0.175cm]T7.east) {};
\path (L7) edge [->,out=40,in=320,loop,min distance=10mm] (L7);
\node[anchor=west] at ([xshift=1.15cm]T7.east) {$\pi_2$};
\node (L8) at ([xshift=0.175cm]T8.east) {};
\path (L8) edge [->,out=40,in=320,loop,min distance=10mm] (L8);
\node[anchor=west] at ([xshift=1.15cm]T8.east) {$\pi_3$};
\end{tikzpicture}
\endgroup
\caption{The $\pi_i$-actions on $\StdSVT_4((2,1))$.}
\label{fig: svt action 214}
\end{figure}
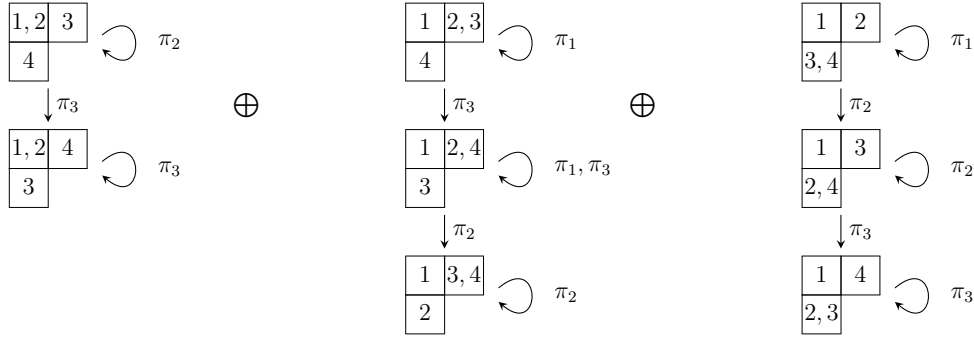
\end{example}

\section[Interval structure on each equivalence class]{Interval structure on each class in $\calE_{\lambda,m}$}
\label{sec: WBIM decomposition}

Throughout this section, fix $E \in \calE_{\lambda,m}$.
We will show that $\scrG_E$ is isomorphic to a weak Bruhat interval module.
Following the general strategy of \cite[Sections~4 and 5]{KY24}, we first find the unique source and sink tableaux in $E$ and then identify the tableaux in $E$ with the permutations in a left weak Bruhat interval.

\subsection[Source and sink tableaux]{Source and sink tableaux in $E$}
\label{subsec: source and sink}

\begin{definition}\label{def: source sink}
Let $T \in E$.
\begin{enumerate}[label = {\rm (\arabic*)}]
\item 
We call $T$ a \emph{source tableau} if there do not exist $T' \in E$ and $1 \le i \le m-1$ such that $T' \ne T$ and $\pi_i \cdot T'=T$.
\item
We call $T$ a \emph{sink tableau} if there do not exist $T' \in E$ and $1 \le i \le m-1$ such that $T' \ne T$ and $\pi_i \cdot T=T'$.
\end{enumerate}
\end{definition}

To prove that $E$ has a unique source tableau and a unique sink tableau,
we define a poset associated with $E$.
For each box $B \in \tyd(\lambda)$, define
\[
d_B:=|T(B)|,
\]
where $T$ is any tableau in $E$.
This is independent of the choice of $T$ by \cref{def: column equivalence}.
For each box $B$ and $1 \le p \le d_B$, let $B^{(p)}$ denote the slot corresponding to the $p$th smallest entry of $T(B)$ for every $T \in E$. 
Let $\mathsf{S}_E$ be the set of all these slots.
Define $P_E=(\mathsf{S}_E,\preceq_{P_E})$, where $\preceq_{P_E}$ is the partial order given by the following conditions.
\begin{enumerate}[label = {\rm (\arabic*)}]
\item For each box $B$, we have
\[
B^{(1)}\prec_{P_E} B^{(2)}\prec_{P_E}\cdots\prec_{P_E} B^{(d_B)}.
\]
\item If $B_1=(r_1,c_1)$ and $B_2=(r_2,c_2)$ are distinct boxes satisfying $r_1\le r_2$ and $c_1\le c_2$, then
\[
B_1^{(p)}\prec_{P_E} B_2^{(q)}
\]
for all $1 \le p \le d_{B_1}$ and $1 \le q \le d_{B_2}$.
\end{enumerate}

We next construct a bijection between $E$ and the set of linear extensions of $P_E$.
For $T \in E$ and $i \in [m]$, let $\sfx_i(T)$ denote the unique slot corresponding to the entry $i$ in $T$.
The row and column inequalities for set-valued tableaux imply that the total order $L_T=(\mathsf{S}_E,\preceq_{L_T})$ given by
\[
\sfx_1(T)\prec_{L_T}\sfx_2(T)\prec_{L_T}\cdots\prec_{L_T}\sfx_m(T)
\]
is a linear extension of $P_E$.
Conversely, let $L=(\mathsf{S}_E,\preceq_L)$ be a linear extension of $P_E$.
List the slots of $\mathsf{S}_E$ as $y_1,y_2,\ldots,y_m$ so that
\[
y_1\prec_L y_2\prec_L\cdots\prec_L y_m.
\]
By the definition of $\preceq_{P_E}$, assigning the integer $i$ to the slot $y_i$ for each $i \in [m]$ gives a unique standard set-valued tableau in $E$.
The two constructions are inverse to each other.
Therefore the map from $E$ to the set of linear extensions of $P_E$ defined by $T\mapsto L_T$ is a bijection.

\begin{example}\label{ex: slot poset}
Let $\lambda=(2,1)$, $m=4$, and let $E$ be the equivalence class containing
\[
\begingroup
\ytableausetup{boxsize=1.7em}
T_1=
\begin{ytableau}
1,2 & 3\\
4
\end{ytableau}
\qquad\text{and}\qquad
T_2=
\begin{ytableau}
1,2 & 4\\
3
\end{ytableau}.
\endgroup
\]
Write $B_1=(1,1)$, $B_2=(1,2)$, and $B_3=(2,1)$.
Then $d_{B_1}=2$ and $d_{B_2}=d_{B_3}=1$, so
\[
\mathsf{S}_E=\{B_1^{(1)},B_1^{(2)},B_2^{(1)},B_3^{(1)}\}.
\]
The only pair of incomparable elements in $P_E$ is $\{B_2^{(1)},B_3^{(1)}\}$.
Thus $P_E$ has exactly two linear extensions, $L_{T_1}$ and $L_{T_2}$, given by
\[
\begin{aligned}
B_1^{(1)}
&\prec_{L_{T_1}}B_1^{(2)}
\prec_{L_{T_1}}B_2^{(1)}
\prec_{L_{T_1}}B_3^{(1)}
\quad\text{and}\quad
B_1^{(1)}
&\prec_{L_{T_2}}B_1^{(2)}
\prec_{L_{T_2}}B_3^{(1)}
\prec_{L_{T_2}}B_2^{(1)}.
\end{aligned}
\]
\end{example}

For $B=(r,c)$ and $1 \le p \le d_B$, define $\col(B^{(p)}):=c$.
The following lemma characterizes source and sink tableaux in terms of the column indices of consecutive slots in $L_T$.

\begin{lemma}\label{lem: source sink characterization}
Let $T \in E$.
Then the following hold.
\begin{enumerate}[label = {\rm (\arabic*)}]
\item The tableau $T$ is a source tableau if and only if
\[
\col(\sfx_i(T))>\col(\sfx_{i+1}(T))
\]
for all $1 \le i \le m-1$ such that $\sfx_i(T)$ and $\sfx_{i+1}(T)$ are incomparable in $P_E$.
\item The tableau $T$ is a sink tableau if and only if
\[
\col(\sfx_i(T))\le\col(\sfx_{i+1}(T))
\]
for all $1 \le i \le m-1$.
\end{enumerate}
\end{lemma}

\begin{proof}
\noindent{\rm (1)}
To prove the ``only if'' direction, suppose that $T$ is a source tableau.
Fix $1 \le i \le m-1$ such that $\sfx_i(T)$ and $\sfx_{i+1}(T)$ are incomparable.
Since $\sfx_i(T)$ and $\sfx_{i+1}(T)$ are incomparable, interchanging them in $L_T$ gives another linear extension $L'$ of $P_E$.
Let $T' \in E$ be the tableau corresponding to $L'$.
Since slots in the same column are comparable in $P_E$, we have $\col(\sfx_i(T)) \neq \col(\sfx_{i+1}(T))$.
If $\col(\sfx_i(T)) < \col(\sfx_{i+1}(T))$, then $\col(\sfx_i(T')) > \col(\sfx_{i+1}(T'))$, and hence $\pi_i \cdot T'=T$.
This contradicts the assumption that $T$ is a source tableau.
Therefore $\col(\sfx_i(T))>\col(\sfx_{i+1}(T))$.

We prove the ``if'' direction by contraposition.
Suppose that $T$ is not a source tableau.
By \cref{def: source sink}(1), there exist $T' \in E$ and $1 \le i \le m-1$ such that $T'\ne T$ and $\pi_i \cdot T'=T$.
The equality $\pi_i \cdot T'=T$, together with $T'\ne T$, implies that $T'=s_i \cdot T$, so the slots $\sfx_i(T)$ and $\sfx_{i+1}(T)$ occur in opposite orders in $L_T$ and $L_{T'}$.
Since $L_T$ and $L_{T'}$ are linear extensions of $P_E$, $\sfx_i(T)$ and $\sfx_{i+1}(T)$ are incomparable in $P_E$.
Moreover, since $\pi_i \cdot T'=T$ with $T'\ne T$, we have $\col(\sfx_{i+1}(T')) < \col(\sfx_i(T'))$.
Combining this inequality with the equality $T' = s_i \cdot T$ yields
\[
\col(\sfx_i(T)) = \col(\sfx_{i+1}(T')) < \col(\sfx_i(T')) = \col(\sfx_{i+1}(T)).
\]
This proves the ``if'' direction by contraposition.
\medskip

\noindent{\rm (2)}
To prove the ``only if'' direction, suppose that $T$ is a sink tableau.
Fix $1 \le i \le m-1$.
If $\col(\sfx_i(T))>\col(\sfx_{i+1}(T))$, then $T':=s_i \cdot T$ is another tableau in $E$, and $\pi_i \cdot T=T'$.
This contradicts the assumption that $T$ is a sink tableau.
Therefore $\col(\sfx_i(T))\le\col(\sfx_{i+1}(T))$.

We prove the ``if'' direction by contraposition.
Suppose that $T$ is not a sink tableau.
By \cref{def: source sink}(2), there exist $T' \in E$ and $1 \le i \le m-1$ such that $T'\ne T$ and $\pi_i \cdot T=T'$.
The equality $\pi_i \cdot T=T'$, together with $T'\ne T$, implies that
\[
\col(\sfx_i(T))>\col(\sfx_{i+1}(T)).
\]
This proves the ``if'' direction by contraposition.
\end{proof}

We next define two special tableaux in $E$ by the following algorithm.
\begin{algorithm}\label{alg: extremal tableaux}
Let $E \in \calE_{\lambda,m}$.
\begin{enumerate}[label = {\it Step {\rm \arabic*:}}]
\item 
Set $i = 1$ and $U_0 = V_0 = \emptyset$.
\item
Among the minimal elements of $P_E \setminus U_{i-1}$, let $u_i$ be the slot with the largest column index.
Among the minimal elements of $P_E \setminus V_{i-1}$, let $v_i$ be the slot with the smallest column index.
\item
If $i < m$, set $U_{i} = U_{i-1} \cup \{u_i\}$, $V_{i} = V_{i-1} \cup \{v_i\}$, $i = i+1$, and go to {\it Step 2}.
Otherwise, define $T_E$ and $T'_E$ to be the tableaux in $E$ corresponding to the total orders on $\mathsf{S}_E$ given by
\[
u_1 \prec u_2 \prec \cdots \prec u_m
\quad\text{and}\quad
v_1 \prec' v_2 \prec' \cdots \prec' v_m,
\]
respectively.
\end{enumerate}
\end{algorithm}

The tableaux $T_E$ and $T'_E$ are well-defined.
Indeed, the slots in each column are linearly ordered, so there is at most one minimal remaining slot in each column at every step.
Thus the slot chosen at each step is unique, and repeatedly choosing a minimal remaining slot gives a linear extension of $P_E$.

\begin{example}\label{ex: extremal tableaux}
Let $E$ be the equivalence class in \cref{ex: slot poset}, and let $T_1,T_2 \in E$ be the tableaux given there.
Applying \cref{alg: extremal tableaux}, we see that the first two terms of both sequences are $B_1^{(1)}$ and $B_1^{(2)}$.
After these slots are chosen, the remaining minimal slots are $B_2^{(1)}$ and $B_3^{(1)}$, whose column indices are $2$ and $1$, respectively.
Therefore
\[
(u_1,u_2,u_3,u_4)
=\bigl(B_1^{(1)},B_1^{(2)},B_2^{(1)},B_3^{(1)}\bigr)
\quad\text{and}\quad
(v_1,v_2,v_3,v_4)
=\bigl(B_1^{(1)},B_1^{(2)},B_3^{(1)},B_2^{(1)}\bigr).
\]
Hence $T_E=T_1$ and $T'_E=T_2$.
\end{example}

\begin{theorem}\label{thm: unique source sink}
The tableau $T_E$ is the unique source tableau in $E$, and $T'_E$ is the unique sink tableau in $E$.
\end{theorem}

\begin{proof}
We first prove that $T_E$ is a source tableau.
Let $u_1,u_2,\ldots,u_m$ be the slots defined in \cref{alg: extremal tableaux}, and let $1 \le i \le m-1$ be an index such that $u_i$ and $u_{i+1}$ are incomparable in $P_E$.
Since $u_i$ and $u_{i+1}$ are incomparable, the minimality of $u_{i+1}$ in $P_E \setminus U_i$ implies that $u_{i+1}$ is minimal in $P_E \setminus U_{i-1}$.
In \textit{Step} 2 of \cref{alg: extremal tableaux}, $u_i$ is chosen to have the largest column index among the minimal slots of $P_E \setminus U_{i-1}$, and thus $\col(u_i)>\col(u_{i+1})$.
By \cref{lem: source sink characterization}(1), the tableau $T_E$ is a source tableau.

We next prove the uniqueness of the source tableau.
Let $T \in E$ with $T\ne T_E$, and let $k$ be the smallest index such that $\sfx_k(T)\ne u_k$.
Then $u_k=\sfx_l(T)$ for some $l>k$.
Since $\sfx_j(T)=u_j$ for all $1 \le j<k$, both $u_k$ and $\sfx_k(T)$ are minimal in $P_E \setminus U_{k-1}$.
By the choice of $u_k$, we have
\begin{align}\label{eq: col inequality}
\col(\sfx_k(T)) < \col(u_k).
\end{align}
For $k\le j<l$, the minimality of $u_k$ in $P_E \setminus U_{k-1}$ implies that $\sfx_j(T) \nprec_{P_E} u_k$, while the fact that $\sfx_j(T)$ occurs before $u_k$ in $L_T$ implies that $u_k \nprec_{P_E} \sfx_j(T)$.
It follows that $u_k$ and $\sfx_j(T)$ are incomparable in $P_E$ for all $k\le j<l$.
By \cref{eq: col inequality} and the equality $\sfx_l(T)=u_k$, we can choose an index $j$ such that $k\le j<l$ and
\[
\col(\sfx_j(T))<\col(u_k)\le\col(\sfx_{j+1}(T)).
\]
If $j+1=l$, then $\sfx_j(T)$ and $\sfx_{j+1}(T)=u_k$ are incomparable by the preceding observation.
In this case, by the choice of $j$ and \cref{lem: source sink characterization}(1), $T$ is not a source tableau.
Suppose that $j+1<l$.
Since $\sfx_{j+1}(T)$ and $u_k$ are incomparable, they do not lie in the same column, and hence $\col(u_k)<\col(\sfx_{j+1}(T))$.
Combining the fact that $\sfx_j(T)$ and $\sfx_{j+1}(T)$ are both incomparable with $u_k$ in $P_E$ with the inequalities $\col(\sfx_j(T))<\col(u_k)<\col(\sfx_{j+1}(T))$, we see that $\sfx_j(T)$ is strictly below and strictly to the left of $u_k$, whereas $\sfx_{j+1}(T)$ is strictly above and strictly to the right of $u_k$.
Therefore $\sfx_j(T)$ and $\sfx_{j+1}(T)$ are incomparable.
By the choice of $j$ and \cref{lem: source sink characterization}(1), $T$ is not a source tableau.
Thus $T_E$ is the unique source tableau in $E$.

Let us prove that $T'_E$ is a sink tableau.
Let $v_1,v_2,\ldots,v_m$ be the slots defined in \cref{alg: extremal tableaux}, and let $1 \le i \le m-1$.
If $v_i$ and $v_{i+1}$ are incomparable in $P_E$, then the minimality of $v_{i+1}$ in $P_E \setminus V_i$ implies that $v_{i+1}$ is minimal in $P_E \setminus V_{i-1}$.
In \textit{Step} 2 of \cref{alg: extremal tableaux}, $v_i$ is chosen to have the smallest column index among the minimal slots of $P_E \setminus V_{i-1}$, and thus $\col(v_i)<\col(v_{i+1})$.
If $v_i$ and $v_{i+1}$ are comparable, then the minimality of $v_i$ in $P_E \setminus V_{i-1}$ implies that $v_i \prec_{P_E} v_{i+1}$.
It follows that $\col(v_i)\le\col(v_{i+1})$.
By \cref{lem: source sink characterization}(2), the tableau $T'_E$ is a sink tableau.

We finally prove the uniqueness of the sink tableau.
Let $T \in E$ with $T\ne T'_E$, and let $k$ be the smallest index such that $\sfx_k(T)\ne v_k$.
Then $v_k=\sfx_l(T)$ for some $l>k$.
Since $\sfx_j(T)=v_j$ for all $1 \le j<k$, both $v_k$ and $\sfx_k(T)$ are minimal in $P_E \setminus V_{k-1}$.
The choice of $v_k$ gives $\col(\sfx_k(T))>\col(v_k)$.
By this inequality and the equality $\sfx_l(T)=v_k$, we can choose an index $j$ such that $k\le j<l$ and $\col(\sfx_j(T))>\col(v_k)\ge\col(\sfx_{j+1}(T))$.
It follows from \cref{lem: source sink characterization}(2) that $T$ is not a sink tableau.
Thus $T'_E$ is the unique sink tableau in $E$.
\end{proof}

\subsection[A weak Bruhat interval module structure]{A weak Bruhat interval module structure on $\scrG_E$}

To begin with, we assign a permutation to each $T \in E$.

\begin{definition}\label{def: column reading}
For $T \in E$, $\sfread(T)$ is defined to be the word obtained from $T$ by reading the entries from top to bottom starting with the rightmost column and in increasing order within each box.
\end{definition}

We identify $\sfread(T)$ with the permutation in $\SG_m$ written in one-line notation.

\begin{example}\label{ex: reading word example}
Let
\[\ytableausetup{boxsize=2em}
T=
\begin{ytableau}
\scriptstyle 1,2 & \scriptstyle 3 & \scriptstyle 4,5,6 & \scriptstyle 7,9,11\\
\scriptstyle 8,10
\end{ytableau}
\in \StdSVT_{11}((4,1))
\]
be the tableau in \cref{ex: descent reading example}.
Reading the fourth, third, second, and first columns of $T$ gives the words $7\,9\,11$, $4\,5\,6$, $3$, and $1\,2\,8\,10$, respectively.
Thus,
\[
\sfread(T)=7\,9\,11\,4\,5\,6\,3\,1\,2\,8\,10 \in \SG_{11}.
\]
\end{example}

\begin{remark}
For $T \in E$, the reading word $w_T$ used for the $F$-expansion is different from $\sfread(T)$ in general.
For example, let $T$ be as in \cref{ex: reading word example}.
Then
\[
w_T=10\,8\,11\,9\,6\,5\,2\,1\,3\,4\,7 \quad \text{and} \quad
\sfread(T)=7\,9\,11\,4\,5\,6\,3\,1\,2\,8\,10.
\]
The words $w_T$ and $\sfread(T)$ are computed in \cref{ex: descent reading example} and \cref{ex: reading word example}, respectively.
\end{remark}

The following lemma shows that the reading defined in \cref{def: column reading} is compatible with descent sets.

\begin{lemma}\label{lem: column reading properties}
For $T \in E$, $\Des(T) = [m-1] \setminus \DesL(\sfread(T))$.
\end{lemma}

\begin{proof}
Let $i \in [m-1]$.
The word $\sfread(T)$ is obtained by reading the columns of $T$ from right to left and listing the entries in each column in increasing order.
This means that $i+1$ appears before $i$ in $\sfread(T)$ if and only if $\col_T(i)<\col_T(i+1)$.
Thus, $\Des(T) = [m-1] \setminus \DesL(\sfread(T))$.
\end{proof}

Let $\preceq_E$ be a relation on $E$ defined by 
\begin{align*}
T \preceq_E U
\quad \text{if and only if} \quad 
U = \pi_\sigma \cdot T \text{ for some } \sigma \in \SG_m.
\end{align*}
The following proposition shows that $(E,\preceq_E)$ is a poset isomorphic to a left weak Bruhat interval.

\begin{proposition}\label{prop: class is weak interval}
The following hold.
\begin{enumerate}[label = {\rm (\arabic*)}]
\item
If $T,U \in E$ and $T \preceq_E U$, then $\sfread(T) \preceq_L \sfread(U)$.

\item
The poset $(E,\preceq_E)$ is isomorphic to $([\sfread(T_E),\sfread(T'_E)]_L,\preceq_L)$.
\end{enumerate}
\end{proposition}

\begin{proof}
\noindent{\rm (1)}
Let $T,U \in E$ with $T \preceq_E U$.
By the definition of $\preceq_E$, there exists $\sigma \in \SG_m$ such that $U = \pi_\sigma \cdot T$.
Write a reduced expression for $\sigma$ as $s_{i_k}\cdots s_{i_2}s_{i_1}$, and define 
\[
T_0:=T 
\quad \text{and} \quad
T_j:=\pi_{i_j} \cdot T_{j-1}
\quad\text{for } 1 \le j \le k.
\]
Then $T_k = U$.
Since $U \ne 0$, $T_j \ne 0$ for all $1 \le j \le k$, so each $T_j$ belongs to $E$.
Fix $1 \le j \le k$.
If $T_j = T_{j-1}$, then $\sfread(T_j) = \sfread(T_{j-1})$.
Suppose that $T_j \ne T_{j-1}$.
Then 
\[
T_j = s_{i_j} \cdot T_{j-1}
\quad \text{and} \quad
\col_{T_{j-1}}(i_j) > \col_{T_{j-1}}(i_j+1).
\]
We have $i_j \in \Des(T_{j-1})$, so \cref{lem: column reading properties} implies that $i_j \notin \DesL(\sfread(T_{j-1}))$.
We also have $\sfread(T_j)=s_{i_j}\sfread(T_{j-1})$.
Combining these two observations yields $\sfread(T_{j-1}) \preceq_L^c \sfread(T_j)$.
Therefore,
\[
\sfread(T) = \sfread(T_0)
\preceq_L \sfread(T_1)
\preceq_L \cdots
\preceq_L \sfread(T_k) = \sfread(U).
\]

\noindent{\rm (2)}
For any $T \in E$, we have $T = \pi_{\mathrm{id}} \cdot T$, so $\preceq_E$ is reflexive.
For transitivity, suppose that $T,T',T'' \in E$ and $T \preceq_E T'$ and $T' \preceq_E T''$.
Then there exist $\sigma,\gamma \in \SG_m$ such that $T' = \pi_\sigma \cdot T$ and $T'' = \pi_\gamma \cdot T'$.
Since $\pi_\gamma\pi_\sigma = \pi_\rho$ for some $\rho \in \SG_m$, we have $T'' = \pi_\rho \cdot T$, which proves that $\preceq_E$ is transitive.
For antisymmetry, suppose that $U, U' \in E$ and $U \preceq_E U'$ and $U' \preceq_E U$.
By (1), we have $\sfread(U) \preceq_L \sfread(U')$ and $\sfread(U') \preceq_L \sfread(U)$.
The antisymmetry of $\preceq_L$ gives $\sfread(U) = \sfread(U')$.
Since the map $E \to \SG_m$ sending $T$ to $\sfread(T)$ is injective, we have $U = U'$.
Therefore, $\preceq_E$ is a partial order on $E$.

Now, let us prove that the map
\[
\phi_E:E\to[\sfread(T_E),\sfread(T'_E)]_L,
\quad
T \mapsto \sfread(T)
\]
is a poset isomorphism.
If $|E|=1$, the assertion is immediate.
Suppose that $|E|>1$.

By the definition of $\preceq_E$, the source and sink tableaux are the minimal and maximal elements of $(E,\preceq_E)$, respectively.
Since $E$ is finite, \cref{thm: unique source sink} implies that
\[
T_E\preceq_E T\preceq_E T'_E
\quad \text{for all } T \in E.
\]
Applying (1) to these inequalities, we have
\[
\sfread(T_E) \preceq_L \sfread(T) \preceq_L \sfread(T'_E) 
\quad \text{for all } T \in E.
\]
In particular,
\[
\{\sfread(T)\mid T \in E\}
\subseteq
[\sfread(T_E),\sfread(T'_E)]_L.
\]
This proves that $\phi_E$ is well defined.

Clearly, the map $\phi_E$ is injective, and (1) shows that it is order-preserving.
It remains to show that $\phi_E$ is surjective and that for $T,U \in E$, if $\sfread(T)\preceq_L\sfread(U)$, then $T\preceq_E U$.

We first prove that $\phi_E$ is surjective.
Let $\gamma \in [\sfread(T_E),\sfread(T'_E)]_L$ and set $\sigma:=\sfread(T_E)$.
Recall that for $T \in E$ and $1 \le i \le m$, $\sfx_i(T)$ denotes the slot corresponding to the entry $i$ in $T$.
Define a filling $T_\gamma$ by assigning the integer $\gamma(k)$ to the slot $\sfx_{\sigma(k)}(T_E)$ for each $1 \le k \le m$.
To show that $T_\gamma \in E$, we recall inversion sets.
For $\tau \in \SG_m$, let
\[
\operatorname{Inv}(\tau):=\{(a,b)\mid 1 \le a < b \le m,\ \tau(a) > \tau(b)\}.
\]
It is well known that $\tau\preceq_L\rho$ if and only if $\operatorname{Inv}(\tau)\subseteq\operatorname{Inv}(\rho)$; see, for example, \cite[Proposition~3.1.3]{BB05}.
Let $1 \le i, j \le m$ be such that $\sfx_i(T_E)\prec_{P_E}\sfx_j(T_E)$, and set $a:=\sigma^{-1}(i)$ and $b:=\sigma^{-1}(j)$.
For any $T \in E$, the entries assigned to $\sfx_i(T_E)$ and $\sfx_j(T_E)$ occur in $\sfread(T)$ at positions $a$ and $b$, respectively, so the relation $\sfx_i(T_E)\prec_{P_E}\sfx_j(T_E)$ implies that
\[
\sfread(T)(a)<\sfread(T)(b).
\]
It follows that $(a,b)\notin\operatorname{Inv}(\sfread(T'_E))$ when $a<b$, whereas $(b,a)\in\operatorname{Inv}(\sfread(T_E))$ when $b<a$.
Combining this with the inclusions
\[
\operatorname{Inv}(\sfread(T_E))
\subseteq \operatorname{Inv}(\gamma)
\subseteq \operatorname{Inv}(\sfread(T'_E))
\]
yields $\gamma(a)<\gamma(b)$.
In other words, the entry of $T_\gamma$ in $\sfx_i(T_E)$ is smaller than the entry of $T_\gamma$ in $\sfx_j(T_E)$.
Since $i$ and $j$ were arbitrary with $\sfx_i(T_E)\prec_{P_E}\sfx_j(T_E)$, the entries of $T_\gamma$ increase along every strict relation in $P_E$, and hence $T_\gamma \in E$.
By construction, $\sfread(T_\gamma)=\gamma$, which proves that $\phi_E$ is surjective.

We next prove that for $T,U \in E$, if $\sfread(T)\preceq_L\sfread(U)$, then $T\preceq_E U$.
Let $T,U \in E$ be such that $\sfread(T)\preceq_L\sfread(U)$.
If $T=U$, then $T\preceq_E U$ by the reflexivity of $\preceq_E$.
Suppose that $T\ne U$.
Choose a saturated chain
\[
\sfread(T)=\gamma_0\preceq_L^c\gamma_1\preceq_L^c\cdots\preceq_L^c\gamma_k=\sfread(U).
\]
For each $0 \le j \le k$, we have
\[
\sfread(T_E)\preceq_L\sfread(T)\preceq_L\gamma_j
\preceq_L\sfread(U)\preceq_L\sfread(T'_E).
\]
It follows that $\gamma_j$ belongs to $[\sfread(T_E),\sfread(T'_E)]_L$.
Since $\phi_E$ is bijective, for each $0 \le j \le k$, there exists a unique tableau $T_j \in E$ such that $\sfread(T_j)=\gamma_j$.
In particular, $T_0=T$ and $T_k=U$.
Fix $j \in [k]$, and let $i_j \in [m-1]$ be the index such that $\gamma_j = s_{i_j}\gamma_{j-1}$.
Since $i_j \notin \DesL(\gamma_{j-1})$, it follows from \cref{lem: column reading properties} that
\[
\col_{T_{j-1}}(i_j)\ge\col_{T_{j-1}}(i_j+1).
\]
Suppose that $\col_{T_{j-1}}(i_j) = \col_{T_{j-1}}(i_j+1)$.
Then $\sfx_{i_j}(T_{j-1})$ and $\sfx_{i_j+1}(T_{j-1})$ are comparable in $P_E$.
By the definition of $L_{T_{j-1}}$, we have
$\sfx_{i_j}(T_{j-1}) \prec_{L_{T_{j-1}}} \sfx_{i_j+1}(T_{j-1})$.
Since $L_{T_{j-1}}$ is a linear extension of $P_E$, these observations imply that
\[
\sfx_{i_j}(T_{j-1})\prec_{P_E}\sfx_{i_j+1}(T_{j-1}).
\] 
On the other hand, $\sfx_{i_j}(T_{j-1})$ and $\sfx_{i_j+1}(T_{j-1})$ occur in opposite orders in $L_{T_{j-1}}$ and $L_{T_j}$.
This contradicts the fact that $L_{T_j}$ is a linear extension of $P_E$.
It follows that $\col_{T_{j-1}}(i_j)>\col_{T_{j-1}}(i_j+1)$, and hence
\[
\pi_{i_j}\cdot T_{j-1}=s_{i_j}\cdot T_{j-1}.
\]
Moreover,
\[
\sfread(\pi_{i_j} \cdot T_{j-1})
= \sfread(s_{i_j}\cdot T_{j-1})
= s_{i_j}\gamma_{j-1}
= \gamma_j
= \sfread(T_j).
\]
Since both $\pi_{i_j}\cdot T_{j-1}$ and $T_j$ belong to $E$, the injectivity of $\phi_E$ implies that $\pi_{i_j}\cdot T_{j-1}=T_j$.
Iterating the equalities $\pi_{i_j}\cdot T_{j-1}=T_j$ for $j=1,2,\ldots,k$, we have
\[
U=\pi_{i_k}\pi_{i_{k-1}}\cdots\pi_{i_1}\cdot T
=\pi_\xi\cdot T \quad \text{for some } \xi \in \SG_m.
\]
This implies that $T\preceq_E U$.

Therefore, $\phi_E$ is a poset isomorphism.
\end{proof}

We now state the main result of this section.

\begin{theorem}\label{thm: component interval module}
For every $E \in \calE_{\lambda,m}$, we have an $H_m(0)$-module isomorphism
\[
\scrG_E\cong\sfB(\sfread(T_E),\sfread(T'_E)).
\]
\end{theorem}

\begin{proof}
By \cref{prop: class is weak interval}, the linear map
\[
\Phi_E:\scrG_E\rightarrow\sfB(\sfread(T_E),\sfread(T'_E)),
\quad
T\mapsto \sfread(T)
\]
is a vector space isomorphism.
It remains to show that $\Phi_E$ commutes with the $H_m(0)$-action.

Let $T \in E$ and $1 \le i \le m-1$.
We consider three cases according to the value of $\pi_i \cdot T$.
\medskip

\noindent
{\bf Case 1: $\pi_i \cdot T=T$.}
By \cref{eq: svt action}, we have $\col_T(i)<\col_T(i+1)$.
It follows from \cref{lem: column reading properties} that $i \in \DesL(\sfread(T))$.
Thus,
\[
\Phi_E(\pi_i \cdot T)=\sfread(T)=\pi_i \cdot \sfread(T).
\]

\noindent
{\bf Case 2: $\pi_i \cdot T=0$.}
By \cref{eq: svt action}, we have $\col_T(i)=\col_T(i+1)$.
It follows from \cref{lem: column reading properties} that $i \notin \DesL(\sfread(T))$.
The slots $\sfx_i(T)$ and $\sfx_{i+1}(T)$ are comparable in $P_E$.
By the definition of $L_T$, we have $\sfx_i(T)\prec_{L_T}\sfx_{i+1}(T)$.
Since $L_T$ is a linear extension of $P_E$, these observations imply that $\sfx_i(T)\prec_{P_E}\sfx_{i+1}(T)$.
Suppose that $s_i \, \sfread(T) \in [\sfread(T_E),\sfread(T'_E)]_L$.
By \cref{prop: class is weak interval}, there exists $U \in E$ such that $\sfread(U)=s_i \, \sfread(T)$.
Then we have $\sfx_{i+1}(T)\prec_{L_U}\sfx_i(T)$.
Since $L_U$ is a linear extension of $P_E$, this is a contradiction.
It follows that $s_i \, \sfread(T)\notin[\sfread(T_E),\sfread(T'_E)]_L$.
Thus,
\[
\Phi_E(\pi_i \cdot T)=0=\pi_i \cdot \sfread(T).
\]

\noindent
{\bf Case 3: $\pi_i \cdot T = s_i \cdot T$.}
By \cref{eq: svt action}, we have $\col_T(i)>\col_T(i+1)$.
It follows from \cref{lem: column reading properties} that $i \notin \DesL(\sfread(T))$.
Since $s_i \cdot T \in E$, \cref{prop: class is weak interval} implies $
\sfread(s_i \cdot T)\in[\sfread(T_E),\sfread(T'_E)]_L$.
Thus,
\[
\Phi_E(\pi_i \cdot T)
=\sfread(s_i \cdot T)
=s_i \, \sfread(T)
=\pi_i \cdot \sfread(T).
\]

Therefore, $\Phi_E$ is an $H_m(0)$-module isomorphism.
\end{proof}

\begin{example}\label{ex: column reading permutation}
Let $E \in \calE_{(3,2),8}$ be the equivalence class such that
\[
c_1(T)=(2,1),\qquad c_2(T)=(2,2),\qquad\text{and}\qquad c_3(T)=(1)
\]
for $T \in E$.
Applying \cref{alg: extremal tableaux}, we have
\[
\ytableausetup{boxsize=1.7em}
T_E=
\begin{ytableau}
 1,2 & 3,4 & 5\\
 6 & 7,8
\end{ytableau}
\quad\text{and}\quad
T'_E=
\begin{ytableau}
1,2 & 4,5 & 8\\
3 & 6,7
\end{ytableau}\;.
\]
Their readings are $\sfread(T_E)=53478126$ and $\sfread(T'_E)=84567123$.
By \cref{thm: component interval module},
\[
\scrG_E\cong\sfB(53478126,84567123).
\]
We illustrate the $H_8(0)$-action on this interval in \cref{fig: interval action 328}, where the zero actions are omitted.
\end{example}

\begin{figure}[h]
\centering
\ytableausetup{boxsize=1.6em}
\begin{tikzpicture}[>=stealth,scale=1,transform shape,
  pair/.style={minimum width=3.1cm,minimum height=1.45cm,inner sep=0pt},
  permutation/.style={inner sep=1pt,font=\scriptsize},
  edge/.style={->,shorten >=4pt,shorten <=4pt},
  edgelabel/.style={fill=white,inner sep=1pt,font=\scriptsize},
  looplabel/.style={anchor=west,inner sep=0pt,font=\scriptsize}]

\node[pair] (V1) at (0,0) {};
\node[permutation] (W1) at ([yshift=0.28cm]V1.center) {$53478126$};

\node[pair] (V2) at (0,-1.5) {};
\node[permutation] (W2) at ([yshift=0.28cm]V2.center) {$63478125$};

\node[pair] (V3) at (-2,-3) {};
\node[permutation] (W3) at ([yshift=0.28cm]V3.center) {$63578124$};

\node[pair] (V4) at (2,-3) {};
\node[permutation] (W4) at ([yshift=0.28cm]V4.center) {$73468125$};

\node[pair] (V5) at (-4,-4.5) {};
\node[permutation] (W5) at ([yshift=0.28cm]V5.center) {$64578123$};

\node[pair] (V6) at (0,-4.5) {};
\node[permutation] (W6) at ([yshift=0.28cm]V6.center) {$73568124$};

\node[pair] (V7) at (4,-4.5) {};
\node[permutation] (W7) at ([yshift=0.28cm]V7.center) {$83467125$};

\node[pair] (V8) at (-2,-6) {};
\node[permutation] (W8) at ([yshift=0.28cm]V8.center) {$74568123$};

\node[pair] (V9) at (2,-6) {};
\node[permutation] (W9) at ([yshift=0.28cm]V9.center) {$83567124$};

\node[pair] (V10) at (0,-7.5) {};
\node[permutation] (W10) at ([yshift=0.28cm]V10.center) {$84567123$};

\draw[edge] (W1) -- node[edgelabel,right, xshift=0.07cm] {$\pi_5$} (W2);
\draw[edge] (W2) -- node[edgelabel,above left] {$\pi_4$} (W3);
\draw[edge] (W2) -- node[edgelabel,above right] {$\pi_6$} (W4);
\draw[edge] (W3) -- node[edgelabel,above left] {$\pi_3$} (W5);
\draw[edge] (W3) -- node[edgelabel,above right] {$\pi_6$} (W6);
\draw[edge] (W4) -- node[edgelabel,above left] {$\pi_4$} (W6);
\draw[edge] (W4) -- node[edgelabel,above right] {$\pi_7$} (W7);
\draw[edge] (W5) -- node[edgelabel,above right] {$\pi_6$} (W8);
\draw[edge] (W6) -- node[edgelabel,above left] {$\pi_3$} (W8);
\draw[edge] (W6) -- node[edgelabel,above right] {$\pi_7$} (W9);
\draw[edge] (W7) -- node[edgelabel,above left] {$\pi_4$} (W9);
\draw[edge] (W8) -- node[edgelabel,above right] {$\pi_7$} (W10);
\draw[edge] (W9) -- node[edgelabel,above left] {$\pi_3$} (W10);

\node (L1) at ([xshift=-0.75cm,yshift=0.28cm]V1.east) {};
\path (L1) edge [->,out=40,in=320,loop,min distance=8mm] (L1);
\node[looplabel] at ([xshift=-0.05cm,yshift=0.28cm]V1.east) {$\pi_2,\pi_4,\pi_6$};
\node (L2) at ([xshift=-0.75cm,yshift=0.28cm]V2.east) {};
\path (L2) edge [->,out=40,in=320,loop,min distance=8mm] (L2);
\node[looplabel] at ([xshift=-0.05cm,yshift=0.28cm]V2.east) {$\pi_2,\pi_5$};
\node (L3) at ([xshift=-0.75cm,yshift=0.28cm]V3.east) {};
\path (L3) edge [->,out=40,in=320,loop,min distance=8mm] (L3);
\node[looplabel] at ([xshift=-0.05cm,yshift=0.28cm]V3.east) {$\pi_2,\pi_4,\pi_5$};
\node (L4) at ([xshift=-0.75cm,yshift=0.28cm]V4.east) {};
\path (L4) edge [->,out=40,in=320,loop,min distance=8mm] (L4);
\node[looplabel] at ([xshift=-0.05cm,yshift=0.28cm]V4.east) {$\pi_2,\pi_5,\pi_6$};
\node (L5) at ([xshift=-0.75cm,yshift=0.28cm]V5.east) {};
\path (L5) edge [->,out=40,in=320,loop,min distance=8mm] (L5);
\node[looplabel] at ([xshift=-0.05cm,yshift=0.28cm]V5.east) {$\pi_3,\pi_5$};
\node (L6) at ([xshift=-0.75cm,yshift=0.28cm]V6.east) {};
\path (L6) edge [->,out=40,in=320,loop,min distance=8mm] (L6);
\node[looplabel] at ([xshift=-0.05cm,yshift=0.28cm]V6.east) {$\pi_2,\pi_4,\pi_6$};
\node (L7) at ([xshift=-0.75cm,yshift=0.28cm]V7.east) {};
\path (L7) edge [->,out=40,in=320,loop,min distance=8mm] (L7);
\node[looplabel] at ([xshift=-0.05cm,yshift=0.28cm]V7.east) {$\pi_2,\pi_5,\pi_7$};
\node (L8) at ([xshift=-0.75cm,yshift=0.28cm]V8.east) {};
\path (L8) edge [->,out=40,in=320,loop,min distance=8mm] (L8);
\node[looplabel] at ([xshift=-0.05cm,yshift=0.28cm]V8.east) {$\pi_3,\pi_6$};
\node (L9) at ([xshift=-0.75cm,yshift=0.28cm]V9.east) {};
\path (L9) edge [->,out=40,in=320,loop,min distance=8mm] (L9);
\node[looplabel] at ([xshift=-0.05cm,yshift=0.28cm]V9.east) {$\pi_2,\pi_4,\pi_7$};
\node (L10) at ([xshift=-0.75cm,yshift=0.28cm]V10.east) {};
\path (L10) edge [->,out=40,in=320,loop,min distance=8mm] (L10);
\node[looplabel] at ([xshift=-0.05cm,yshift=0.28cm]V10.east) {$\pi_3,\pi_7$};
\end{tikzpicture}
\caption{The $H_8(0)$-action on $[\sfread(T_E),\sfread(T'_E)]_L$.}
\label{fig: interval action 328}
\end{figure}

\section*{Declaration of AI Use}

The author used OpenAI Codex with the GPT-5.6 Sol model in conducting this research.
Codex assisted with literature searches, the development of proofs, the construction of examples, and the revision of the manuscript.
The author reviewed and verified all mathematical content and takes full responsibility for the paper.

\bibliographystyle{abbrv}
\bibliography{references}

\end{document}